\documentclass{amsart}
\usepackage{graphicx} 
\usepackage{amsmath}
\usepackage{amssymb}
\usepackage{amsthm}
\usepackage{mathrsfs}
\usepackage{xcolor}
\usepackage{appendix}
\usepackage{url}
\newtheorem{theorem}{Theorem}
\newtheorem{lemma}{Lemma}[section]
\newtheorem{corollary}[lemma]{Corollary}  
\newtheorem{proposition}[theorem]{Proposition}
\newtheorem{conjecture}{Conjecture}

\newtheorem{example}{Example}
\newtheorem{definition}{Definition}

\title[A pluricomplex error-function kernel at the edge]{A pluricomplex error-function kernel at the edge of polynomial Bergman kernels}

\author[L.D. Molag]{Leslie Molag}
\address{L.D. Molag: Mathematics Department \\ Carlos III University of Madrid\\ Avda. de la Universidad, 30. 28911 Leganés\\ Spain}
\email{lmolag@math.uc3m.es }

\begin{document}

\maketitle

\begin{center}
\textit{Dedicated to Gernot Akemann on the occasion of his 60th birthday}
\end{center}

\begin{abstract}
We consider polynomial Bergman kernels with respect to exponentially varying weights $e^{-n \mathscr Q(z)}$ depending on a potential $\mathscr Q:\mathbb C^d\to\mathbb R$. We use these kernels to construct determinantal point processes on $\mathbb C^d$. Under mild conditions on the potential, the points are known to accumulate on a compact set $S_{\mathscr Q}$ called the droplet. We show that the local behavior of the kernel in the vicinity of the edge $\partial S_{\mathscr Q}$ is described in two different ways by universal limiting kernels. One of these limiting kernels is the error-function kernel, which is ubiquitous in random matrix theory, while the other limiting kernel is a new universal object: a multivariate version of the error-function kernel. We prove the universality in two qualitatively different settings: (i) the tensorized case where $\mathscr Q$ decomposes as a sum of planar potentials, and (ii) the case where $\mathscr Q$ is rotational symmetric. We also explicitly identify the subspace of the Bargmann-Fock space where the multivariate error-function kernel is reproducing. To treat regular edge points that exhibit a certain type of bulk degeneracy, we also find the behavior of the planar kernel with number of terms of order $o(n)$ instead of $n$. Lastly, we prove an edge scaling limit for counting statistics. 
\end{abstract}

\tableofcontents

\section{Introduction}


\subsection{Polynomial Bergman kernels on $\mathbb C^d$}

Consider an exponentially varying weight 
\[
\mathscr W(z) = e^{-n \mathscr Q(z)}
\]
where $\mathscr Q:\mathbb C^d\to \mathbb R$ is called the potential, and $n$ is a positive integer. Under certain growth and regularity conditions we may form the polynomial Bergman kernel with respect to this weight, that is, the reproducing kernel on the space of multivariate complex polynomials $\mathscr P$ of degree $<n$ with respect to the norm
\[
\|\mathscr P\|_{n\mathscr Q}^2 = \int_{\mathbb C^d} |\mathscr P(z)|^2 \mathscr W(z) \, d\omega(z),
\]
where $d\omega(z)=dA(z_1)\cdots dA(z_d)$, and $dA(x+i y)=\pi^{-1} dx \, dy$ is the standard Lebesgue (area) measure on $\mathbb C$ normalized by a factor $\pi$. In this paper, we shall impose the growth condition
\begin{align} \label{eq:growthCond}
\liminf_{|z|\to\infty} \frac{\mathscr Q(z)}{\log |z|^2} > 1.
\end{align}
Assuming that $e^{-n \mathscr Q}$ is also integrable, we may then construct a basis of $n$-dependent polynomials $\{\mathscr P_j(z) : j\in J_n\}$ of total degree $<n$ for some index set $J_n$, satisfying the orthogonality conditions.
\begin{align} \label{eq:defOP}
\int_{\mathbb C^d} \mathscr P_j(z) \overline{\mathscr P_k(z)} \mathscr W(z) \, d\omega(z) = \delta_{j,k}.
\end{align}
Given $\mathscr Q$ and $n$, the polynomial Bergman kernel is unique, and explicitly given by the formula
\[
\boldsymbol k_n(z,w) = \sum_{j\in J_n} \mathscr P_j(z) \overline{\mathscr P_j(w)}, \qquad z,w\in\mathbb C^d.
\]
It is independent of the choice of basis of our orthogonal polynomials. 
A related object that is often considered is the \textit{weighted} polynomial Bergman kernel, defined as
\begin{align*}
    \mathscr K_n(z,w) = \sqrt{\mathscr W(z) \mathscr W(w)}\sum_{j\in J_n} \mathscr P_j(z) \overline{\mathscr P_j(w)}, \qquad z,w\in\mathbb C^d.
\end{align*}
It is the reproducing kernel on the space of weighted polynomials
\[
\mathcal W_n = \{e^{-\frac12 n \mathscr Q} \mathscr P : \mathscr P\in \mathbb C[z], \, \deg \mathscr P<n\}.
\]
In this setup, one may form the determinantal point process (DPP) with joint probability density function proportional to the $N_n^d\times N_n^d$ matrix
\begin{align*}
    \det \left(\mathscr K_n(z_j,z_k)\right)_{j,k\in J_n}, \qquad z_j\in \mathbb C^d \quad \forall j\in J_n,
\end{align*}
where $N_n^d=\binom{n+d-1}{d}$.
With probability $1$ the number of (distinct) points in a configuration of the DPP is
\[
\int_{\mathbb C^d} \mathscr K_n(z,z) \, d\omega(z) = |J_n| = N_n^d.
\]
For large $n$ the number of points behaves like $N_n^d \sim n^d/d!$. DPPs were introduced by Macchi \cite{Macchi1975}. They exhibit a built in repulsion between the points, evident from the determinantal structure.
The density of points is given by the 1-point correlation function $\mathscr K_n(z,z)$. Henceforth, we shall denote the 1-point correlation function by
\[
\mathscr K_n(z) = \mathscr K_n(z,z).
\]
It is sometimes called the Bergman function, and its multiplicative inverse the Christoffel function (in some cases one means the unweighted versions instead). It satisfies the special and very convenient extremal property \cite{Bergman1950}:
\[
\mathscr K_n(z) = \sup_{f\in \mathcal W_n\setminus\{0\}}\frac{\displaystyle |f(z)|^2}{\displaystyle\|f\|_{L^2}^2}.
\]
The exact setting described above was the topic of a paper by Berman \cite{Berman2009}, who derived results with far-reaching consequences. The setting can be extended to complex (Kähler) manifolds \cite{Tian1990, Zelditch1998, ZelditchZhou2004, ShiffmanZelditch2004, Catlin1999, BermanBerndtssonSjostrand2008, Berman2008, Berman2009B, Demailly2012, PokornySinger2014, ComanMarinescu2017, Berman2018} (see, e.g., \cite{Lemoine2022, Ioos2025} for more recent papers), although in this paper we restrict our attention to the pluripotential setting with weighted polynomials on $\mathbb C^d$. 
Under mild conditions on $\mathscr Q$, it is known that the points accumulate on a compact set $S_{\mathscr Q}$. Namely, when $\mathscr Q$ is assumed to be $C^{1,1}$ (and \eqref{eq:growthCond} holds), Berman \cite{Berman2009, Berman2018} proved that there exists a compact set $S_{\mathscr Q}$ such that
\begin{align*}
\lim_{n\to\infty} \frac1{N_n^d} \mathscr K_n(z) = \mathfrak{1}_{S_{\mathscr Q}}(z) d!  \det \partial \bar\partial \mathscr Q(z)
\end{align*}
as $n\to\infty$ in $L^1(\mathbb C^d)$, 
where $\partial\bar\partial\mathscr Q$ denotes the complex Hessian matrix
\begin{align*}
     \partial\bar\partial\mathscr Q(z) = 
    \begin{pmatrix}
        \displaystyle \frac{\partial^2\mathscr Q(z)}{\partial z_1\partial\bar z_1} & \displaystyle\frac{\partial^2\mathscr Q(z)}{\partial z_1\partial\bar z_2} & \hdots & \displaystyle
        \frac{\partial^2\mathscr Q(z)}{\partial z_1\partial\bar z_d}\\
        \displaystyle\frac{\partial^2\mathscr Q(z)}{\partial z_2\partial\bar z_1} & \displaystyle \frac{\partial^2\mathscr Q(z)}{\partial z_2\partial\bar z_2} & \hdots & \displaystyle
        \frac{\partial^2\mathscr Q(z)}{\partial z_2\partial\bar z_d}\\
        \vdots & \vdots & \ddots & \vdots\\
        \displaystyle\frac{\partial^2\mathscr Q(z)}{\partial z_d\partial\bar z_1} & \displaystyle\frac{\partial^2\mathscr Q(z)}{\partial z_d\partial\bar z_2} & \hdots & \displaystyle
        \frac{\partial^2\mathscr Q(z)}{\partial z_d\partial\bar z_d}
    \end{pmatrix},
\end{align*}
where, writing $z_j = x_j + i y_j$, we have
\[
\frac{\partial}{\partial z_j} = \frac12 \left(\frac{\partial}{\partial x_j} - i\frac{\partial}{\partial y_j}\right), \qquad
\frac{\partial}{\partial \bar z_j} = \frac12 \left(\frac{\partial}{\partial x_j} + i\frac{\partial}{\partial y_j}\right).
\]
Equivalently, the measure $\mathscr K_n(z) d\omega(z)$ converges weakly to the measure
\[\mathfrak{1}_{S_{\mathscr Q}}(z) d! \det(\partial \bar\partial \mathscr Q(z)) d\omega(z).\] 
This limiting measure is well-known in pluripotential theory (e.g., see \cite{BedfordTaylor1982, Klimek1991, Demailly2012}) and is called the Monge-Ampère measure.\footnote{Some authors define the Monge-Ampère measure using the $d$-fold wedge product $(d d^c \mathscr Q)^d$.}
We call the compact set $S_{\mathscr Q}$ the \textit{droplet}. 
The interior of the droplet, $\mathring{S}_{\mathscr Q}$, we call the \textit{bulk} (we are deviating slightly from Berman's terminology in \cite{Berman2018} here). The boundary $\partial S_{\mathscr Q}$ is called the \textit{edge}.  As proved by Berman \cite{Berman2009, Berman2018}, under the condition that $\mathscr Q$ is $C^{1,1}$, we equivalently have
\begin{align} \label{eq:ddcQ=ddcChQ}
\mathfrak{1}_{S_{\mathscr Q}}(z) \det \partial \bar\partial \mathscr Q(z)
= \det \partial \bar\partial \check{\mathscr Q}(z),
\end{align}
almost everywhere on $S_{\mathscr Q}$, where the \textit{obstacle function} $\check{\mathscr Q}$ is defined as the pointwise supremum
\begin{align} \label{eq:defPluriObstF}
    \check{\mathscr Q}(z)
    = \sup\{q(z) : q\in\mathcal L(\mathbb C^d), \, q\leq \mathscr Q\},
\end{align}
where $\mathcal L(\mathbb C^d)
$ denotes the Lelong class, consisting of all plurisubharmonic functions $\mathbb C^d\to [-\infty,\infty)$ of logarithmic growth at infinity, 
\[
q(z) \leq \log |z|^2 + \mathcal O(1)
\]
as $|z|\to\infty$. A function $q:\mathbb C^d\to [-\infty,\infty)$ is called plurisubharmonic when it is upper semi-continuous, and either subharmonic or identically $-\infty$ on any restriction to a complex line in $\mathbb C^d$.  We define the \textit{predroplet} as the coincidence set
\[
S_{\mathscr Q}^\star = \{z\in\mathbb C : \check{\mathscr Q}(z)=\mathscr Q(z)\}.
\]
We obviously have $S_{\mathscr Q}\subset S_{\mathscr Q}^\star$. 

For $d=1$, the identity \eqref{eq:ddcQ=ddcChQ} holds almost everywhere on $\mathbb C$, i.e., $\check{\mathscr Q}$ is harmonic outside $S_{\mathscr Q}$. When we explicitly consider the case $d=1$, we shall denote the potential by $Q$ rather than the calligraphic symbol $\mathscr Q$, and $P_j$ denote the (unique) degree $j$ complex polynomials with positive leading coefficient that satisfy the orthogonality relations
\begin{align} \label{eq:orthCondPj}
\int_{\mathbb C} P_j(z) \overline{P_k(z)} e^{-n Q(z)} \, dA(z) = \delta_{j,k}, 
\qquad j,k=0,1,\ldots
\end{align}
The case $d=1$ forms a very active research area. Early works investigating (specifically) the $d=1$ case are \cite{ElbauFelder2005, ZabrodinWiegmann2006, HedenmalmMakarov2013}. The corresponding DPP describes the eigenvalues of random normal matrices (RNM), as well as the location of points of 2D Coulomb gases (for a particular temperature). A related picture, is that it describes incompressible quantum Hall fluids \cite{Wiegmann2002}. Furthermore, for the particular choice $Q(z)=|z|^2$ the model describes noninteracting Fermions in a rotating trap, were the mutual repulsion is caused by the Pauli exclusion principle \cite{LacroixMajumdarSchehr2019}. In these models, one considers random $n\times n$ complex normal matrices $M$ distributed by
\[
\frac1{\mathcal Z_n}  \exp\left(-n\mathrm{Tr} \, Q(M)\right),
\]
for some (planar) potential $Q:\mathbb C\to \mathbb R$, where $\mathcal Z_n$ is the normalization constant, and $\mathrm{Tr} \, Q(M)$ is interpreted as the sum of $Q$ over all eigenvalues of $M$. It turns out that the JPDF takes a particularly nice form in this case: it is of the form
\[
\frac1{Z_n} \prod_{1\leq j<k\leq n} |z_j-z_k|^2 \prod_{j=1}^n e^{-n Q(z_j)},
\]
where $Z_n$ is the normalization constant, and $z_1, \ldots, z_n\in\mathbb C$ are the eigenvalues of $M$.
A standard heuristic continuum limit argument provides us with a potential theoretic minimization problem. Namely, minimize the (energy) functional
\[
J(\mu) = \int_{\mathbb C} \int_{\mathbb C} \log \frac1{|z_j-z_k|} \, d\mu(z) d\mu(w) + \int_{\mathbb C} Q(z) d\mu(z)
\]
over all compactly supported Borel probability measures $\mu$ on $\mathbb C$. Under mild conditions on $Q$ the minimizer $\mu=\sigma_Q$, the \textit{equilibrium measure}, exists. In fact, we know that it is explicitly given by the Monge-Ampère measure
\begin{align*}
d\sigma_Q(z) = \Delta Q(z) \, \mathfrak{1}_{S_Q^\star}(z) \, dA(z),
\end{align*}
For $d>1$ the JPDF of the points has a more complicated form, and this continuum limit argument cannot be applied. In particular, there is no straightforward potential theoretic minimization problem. Interestingly, for $d>1$, the JPDF shows that there is not only mutual repulsion between the points, but there is also an avoidance of certain geometric patterns such as circles. 

\subsection{Local scaling limits}

Berman was able to prove that the local asymptotics around interior (bulk) points in $\mathring S_{\mathscr Q}$ are governed by a multivariate generalization of the complex Ginibre kernel \cite[Theorem 3.9]{Berman2009}, which, for $d=1$ first appeared in \cite{Ginibre1965}. Namely, if one assumes that $\mathscr Q$ is $C^\infty$ in a neighborhood of a bulk point $z_0\in\mathring S_{\mathscr Q}$ and $\partial\bar\partial \mathscr Q(z_0)$ is strictly positive definite, then one finds\footnote{Berman does not present the result explicitly in this form, but the above formula can be extracted from \cite{Berman2009}.}
\begin{multline} \label{eq:bulkLimit}
\lim_{n\to\infty} \frac{c_n(z_0,\xi) \overline{c_n(z_0,\eta)}}{\det n\partial \bar\partial \mathscr Q(z_0)} 
\mathscr K_n\left(z_0+\frac{\xi}{\sqrt{n \partial \bar\partial \mathscr Q(z_0)}}, z_0+\frac{\eta}{\sqrt{n\partial \bar\partial \mathscr Q(z_0)}} \right)\\
= \exp\left(\xi\cdot \eta - \frac{|\xi|^2+|\eta|^2}2\right), \quad \xi, \eta\in\mathbb C^d,
\end{multline}
where $\xi\mapsto c_n(z_0, \xi)$ is a unimodular factor, and $\xi\cdot\eta=\xi_1\overline{\eta_1}+\cdots\xi_d\overline{\eta_d}$ denotes the complex dot (inner) product. Here and henceforth, we use the convention that
\[
\frac{\xi}{\sqrt{\partial\bar\partial \mathscr Q(z_0)}} = (\partial\bar\partial \mathscr Q(z_0))^{-1/2} \xi,
\]
i.e. $(\partial\bar\partial \mathscr Q(z_0))^{-1/2}$ is applied from the left to whatever is in the numerator. 
To abbreviate notation henceforth, we introduce the following definition. In our case $X\subset\mathbb C^d$ always. 

\begin{definition}
Given $f,g : X\times X\to \mathbb C$, we say that $f$ and $g$ \textit{equal up to co-cycles} (on $X\times X$), notation $f\equiv g$, if there exists a unimodular function $c : X\to \mathbb T$ such that $c(z) \overline{c(w)} f(z,w)=g(z,w)$. When $f_n : X\times X\to\mathbb C$ is a sequence, we write
\begin{align*}
\lim_{n\to\infty} f_n \equiv g
\end{align*}
(uniformly) if there exists a sequence $c_n : X\to \mathbb T$ such that (uniformly) 
\begin{align*}
\lim_{n\to\infty} c_n(z) \overline{c_n(w)} f_n(z,w) = g(z,w), \quad \forall (z,w)\in X\times X.
\end{align*}
\end{definition}

Note that, if correlation kernels agree up to co-cycles, they induce the same DPP. 
With this definition we may thus write instead
\begin{multline*}
\lim_{n\to\infty} \frac{1}{\det n\partial \bar\partial \mathscr Q(z_0)} 
\mathscr K_n\left(z_0+\frac{\xi}{\sqrt{n\partial \bar\partial \mathscr Q(z_0)}}, z_0+\frac{\eta}{\sqrt{n\partial \bar\partial \mathscr Q(z_0)}} \right)\\
\equiv \exp\left(\xi\cdot \eta - \frac{|\xi|^2+|\eta|^2}2\right).
\end{multline*}
A limiting behavior of this type is what is known as \textit{Universality}: in a suitably scaled regime, one finds a limiting behavior that is independent of the initial data, i.e., of the potential $\mathscr Q$. The universality class is not limited to our setting or Bermans in \cite{Berman2018}, the limiting kernel is also found for weights $e^{-\mathscr Q}$ in the case that $\mathscr Q$ has positive curvature everywhere (in the more general setting of Kähler manifolds), as implied by the Tian-Zelditch-Catlin expansion.

The limiting kernel factorizes into planar Ginibre kernels and can be considered a pluricomplex version of the ($d=1$) Ginibre kernel, namely
\begin{align*}
\exp\left(\xi\cdot \eta - \frac{|\xi|^2+|\eta|^2}2\right)
= \prod_{k=1}^d \exp\left(\xi_k \overline{\eta_k}-\frac{|\xi_k|^2+|\eta_k|^2}2\right).
\end{align*}
Note that the condition that $\partial\bar\partial \mathscr Q(z_0)$ is strictly positive definite, is equivalent to saying that $\mathscr Q$ is strictly plurisubharmonic on a neighborhood of $z_0$. 
In fact the conditions can be considerably weakened, Berman showed in \cite[Theorem 1.1]{Berman2018} that an analogous statement holds under the condition that $\mathscr Q$ is locally $C^{1,1}$, expressed with the help of the eigenvalues of the complex Hessian in the distributional sense (i.e., the Monge-Ampère operator).

Much less is known concerning scaling limits at the boundary (edge) of the droplet, except for the case $d=1$. In the seminal paper \cite{HedenmalmWennman2021}, it was proved by Hedenmalm and Wennman under mild conditions on $Q:\mathbb C\to\mathbb R$, that for $z_0\in\partial S_{\mathscr Q}$ and $\vec n(z_0)$ the outward unit normal vector at $z_0$ on $\partial S_{\mathscr Q}$, we have
\begin{multline*}
    \lim_{n\to\infty} \frac1{n\Delta Q(z_0)} \mathscr K_n\left(z_0+\frac{\vec n(z_0)\xi}{\sqrt{n\Delta Q(z_0)}}, z_0+\frac{\vec n(z_0)\eta}{\sqrt{n\Delta Q(z_0)}}\right)\\
    \equiv
        \frac12 \exp\left(\xi\overline\eta-\frac{|\xi|^2+|\eta|^2}2\right)
        \mathrm{erfc}\left(\frac{\xi+\overline\eta}{\sqrt 2}\right),
\end{multline*}
locally uniformly for $\xi,\eta\in\mathbb C$ as $n\to\infty$. Here $\Delta=\partial\bar\partial=\frac14 (\partial_x-i\partial_y)(\partial_x+i\partial_y)$ denotes the quarter Laplacian. We define the complementary error-function as
\[
\mathrm{erfc} \, z = \frac{2}{\sqrt\pi} \int_z^\infty e^{-\zeta^2} d\zeta.
\] 
The limiting kernel on the right-hand side is called the error-function (or erfc) kernel, or Faddeeva plasma kernel (who first tabulated it \cite{FaddeyevaTerentev1961}). For explicit models, the limiting kernel was already derived before \cite{ForresterHonner1999, LeeRiser2016}. The error-function kernel does not only occur as a local scaling limit for random normal matrices, Tao and Vu proved that it also shows up in other non-Hermitian random matrix models called independent entry matrices \cite{TaoVu2015}.

\subsection{Local edge universality conjectures}

Concerning the general $d\geq 1$ setting, Berman left ``the case of the
boundary (edge) properties as [a] challenging open problem for the future'' \cite{Berman2018}. The purpose of this paper is to argue that there are two different (though related) universal limiting behaviors to be expected. There has been some progress recently.
The limiting erfc kernel was shown to appear in $d>1$ as well as in \cite{Molag2023} for a one parameter family of potentials $${\mathscr Q(z)=|z|^2-\tau \,\mathrm{Re}(z_1^2+\ldots+z_d^2)},$$ where $\tau\in[0,1)$ is a fixed parameter. For $d=1$ this model, introduced in \cite{Girko1986, SommersCrisantiSompolinskyStein1988}, is a highly researched random matrix model known by the name of (complex) elliptic Ginibre ensemble. For $d>1$ the model was first introduced in \cite{AkemannDuitsMolag2023}. By now, we feel the associated DPP deserves a name and we shall call it the pluripotential elliptic Ginibre ensemble.  In this case the droplet is a hyperellipsoid (or $2d$ dimensional sphere when $\tau=0$). It was shown in \cite{Molag2023} that for $z_0\in\partial S_{\mathscr Q}$ and $\vec n(z_0)\in\mathbb C^d$ the outward unit normal vector at $z_0$ on $\partial S_{\mathscr Q}$ 
\begin{multline*}
    \lim_{n\to\infty} \frac1{n^d} \mathscr K_n\left(z_0+\frac{\vec n(z_0)\xi}{\sqrt{n}}, z_0+\frac{\vec n(z_0)\eta}{\sqrt{n}}\right)\\
    \equiv
        \frac12 \exp\left(\xi\overline\eta-\frac{|\xi|^2+|\eta|^2}2\right)
        \mathrm{erfc}\left(\frac{\xi+\overline\eta}{\sqrt 2}\right),
\end{multline*}
uniformly for $\xi,\eta\in\mathbb C$ with $|\xi|,|\eta|=\mathcal O(n^\nu)$ as $n\to\infty$, for any fixed $\nu\in(0,\frac13)$. 
One of the main contributions of this paper, is to provide strong evidence that this limiting edge behavior is universal. Based on the results in \cite{Molag2023} and the current paper, we expect the following conjecture to hold, which says that the known $d=1$ universality of the error-function kernel extends to $d>1$.

\begin{conjecture} \label{con:1}
Suppose that $\mathscr Q:\mathbb C^d\to\mathbb R$ is $C^2$ and strictly plurisubharmonic.
    Assume furthermore that the droplet $S_{\mathscr Q}$ has a smooth boundary. Let $z_0\in\partial S_{\mathscr Q}$ and denote by $\vec n(z_0)\
    \in\mathbb C^d$ the outward unit normal vector at $z_0$ on $\partial S_{\mathscr Q}$.
Then 
    \begin{multline} \label{eq:con1}
        \lim_{n\to\infty} \frac1{\det n \partial\bar\partial\mathscr Q(z_0)} \mathscr K_n\left(z_0+\frac{\vec n(z_0)\xi}{\sqrt{n \partial\bar\partial\mathscr Q(z_0)}} ,z_0+\frac{\vec n(z_0) \eta}{\sqrt{n \partial\bar\partial\mathscr Q(z_0)}}\right)\\
        \equiv
        \frac12 \exp\left(\xi\overline\eta-\frac{|\xi|^2+|\eta|^2}2\right)
        \mathrm{erfc}\left(\frac{\xi+\overline\eta}{\sqrt 2}\right)
    \end{multline}
    locally uniformly for $\xi,\eta\in\mathbb C$ .
\end{conjecture}

Note that the conditions on $\partial S_{\mathscr Q}$ imply that $z_0$ is a regular boundary point, and $\vec n(z_0)$ to exist. On the diagonal $\xi=\eta$ a similar universal limiting behavior was observed in different but related geometric settings concerning partial Bergman kernels \cite{RossSinger2017, ZelditchZhou2019}.

Conjecture \ref{con:1} is the natural higher-dimensional analogue of the known $d=1$ universal local edge scaling limit. The main novelty of the current paper is that there is a second type of universal behavior to be expected locally at the edge. It originates from the question: if there is a universal limiting object depending on $d$ complex variables in the bulk, namely \eqref{eq:bulkLimit}, would there also be a universal limiting object depending on $d$ complex variables at the edge? At first glance there does not need to be. Once we rescale a bulk point $z_0\in \mathring S_{\mathscr Q}$ using the complex Hessian matrix, in some sense playing the role of curvature, the coordinates near $z_0$ behave isotropically, and we see this in the local scaling limit \eqref{eq:bulkLimit} which only depends on the angle between $\xi$ and $\eta$ (up to a phase factor). On the contrary, when $z_0\in \partial S_{\mathscr Q}$, the situation looks rather anisotropic: even if we use the complex Hessian to rescale our coordinates, there are vectors pointing inside and vectors pointing outside of the droplet, and the kernel behaves quite differently in $\mathring S_{\mathscr Q}$ and $\mathbb C^d\setminus S_{\mathscr Q}$. Indeed, as we shall see, one generally finds a geometric remnant in the limiting behavior in the form of the outward unit normal vector on $\partial S_{\mathscr Q}$ at $z_0$, and thus no apparent universality. Nevertheless, in the models we consider it turns out that one may obtain a universal scaling limit, if one first rotates the coordinates appropriately using a unitary matrix. We give strong evidence for this in Theorem \ref{thm:generalMain} and Theorem \ref{thm:generalMain2} below, and specifically for the rotation argument in a general setting in Proposition \ref{lem:unitVector} below. Our final conjecture is that there is a universal local edge behavior given by a multivariate error-function kernel. 
 

\begin{conjecture} \label{con:2}
Suppose that $\mathscr Q:\mathbb C^d\to\mathbb R$ is $C^2$ and strictly plurisubharmonic.
    Assume furthermore that the droplet $S_{\mathscr Q}$ has a smooth boundary.\\
For any $z_0\in\partial S_{\mathscr Q}$ there is a unitary matrix $\mathscr U(z_0)$ such that 
    \begin{multline} \label{eq:con2}
        \lim_{n\to\infty} \frac1{\det n \partial\bar\partial\mathscr Q(z_0)} \mathscr K_n\left(z_0+\frac{\mathscr U(z_0)\xi}{\sqrt{n \det\partial\bar\partial\mathscr Q(z_0)}} ,z_0+\frac{\mathscr U(z_0) \eta}{\sqrt{n \det\partial\bar\partial\mathscr Q(z_0)}}\right)\\
        \equiv
        \frac12 \exp\left(\xi\cdot\eta-\frac{|\xi|^2+|\eta|^2}2\right)
        \mathrm{erfc}\left(\sum_{k=1}^d \frac{\xi_k+\overline{\eta_k}}{\sqrt {2d}}\right)
    \end{multline}
    locally uniformly for $\xi,\eta\in\mathbb C^d$.
\end{conjecture}

It will be interesting to see if the universality class of this new limiting kernel extends to other models (as in the $d=1$ case).  It is to be expected that some conditions may be weakened, e.g., it is probably enough that $\mathscr Q$ is strictly plurisubharmonic on a neighborhood of $\partial S_{\mathscr Q}$, as long as we impose that $S_{\mathscr Q}=S_{\mathscr Q}^\star$. 

\subsection{Summary of the main results}

    Our main results show that Conjecture \ref{con:1} and Conjecture \ref{con:2} hold for two qualitatively different settings.
    
    \begin{itemize}
    	\item[(i)] The setting where the weight factorizes as a \textit{product of planar weights}, 
\[
\mathscr Q(z) = \sum_{k=1}^d Q_k(z_k),
\]
where for each $k=1,\ldots,d$ we have functions $Q_k:\mathbb C\to \mathbb R$. 
    	\item[(ii)] The setting where the weight is \textit{rotational symmetric}, 
    \[
    \mathscr Q(z)=V(|z|),
    \]
    for some function $V:[0,\infty)\to\mathbb R$. 
    \end{itemize}
It is easy to show that, when $d>1$, only the pluricomplex version of the Ginibre ensemble, corresponding to $\mathscr Q(z)=|z|^2$, is in the intersection of the two settings (up to rescaling). 
We will have to impose some regularity and growth conditions in the two settings. In setting (i) we shall assume that all $Q_k$ are $[0,1]$-admissible. We postpone the exact definition of $[0,1]$-admissibility to Section \ref{sec:2} below (Definition \ref{def:01Adm}), but mention that it is a straightforward generalization of the concept of $\tau$-admissibility introduced in \cite{HedenmalmWennman2021}. 

\begin{theorem} \label{thm:generalMain}
    Suppose that $\mathscr Q:\mathbb C^d\to\mathbb R$ decomposes as a sum of $[0,1]$-admissible planar potentials.
    Assume that the droplet $S_{\mathscr Q}$ has a smooth boundary. Then the following statements are true. 
    \begin{itemize}
    \item[(i)]
    For any $z_0\in\partial S_{\mathscr Q}$ denote by $\vec n(z_0)\
    \in\mathbb C^d$ the outward unit normal vector at $z_0$ on $\partial S_{\mathscr Q}$.
Then we have as $n\to\infty$ that
    \begin{multline*} 
        \frac1{\det n \partial\bar\partial\mathscr Q(z_0)} \mathscr K_n\left(z_0+\frac{\vec n(z_0)\xi}{\sqrt{n \partial\bar\partial\mathscr Q(z_0)}} ,z_0+\frac{\vec n(z_0) \eta}{\sqrt{n \partial\bar\partial\mathscr Q(z_0)}}\right)\\
        \equiv
        \left(1+\mathcal O\left(\frac{\log^3n}{\sqrt n}\right)\right)
        \frac12 \exp\left(\xi\overline\eta-\frac{|\xi|^2+|\eta|^2}2\right)
        \mathrm{erfc}\left(\frac{\xi+\overline\eta}{\sqrt 2}\right)
    \end{multline*}
    uniformly for $z_0\in\partial S_{\mathscr Q}$ and $\xi,\eta\in\mathbb C$ with $|\xi|,|\eta|=\mathcal O(\sqrt{\log n})$.
    \item[(ii)]
    For any $z_0\in\partial S_{\mathscr Q}$ there is a unitary matrix $\mathscr U(z_0)$ such that as $n\to\infty$
    \begin{multline*} 
        \frac1{\det n \partial\bar\partial\mathscr Q(z_0)} \mathscr K_n\left(z_0+\frac{\mathscr U(z_0)\xi}{\sqrt{n \det\partial\bar\partial\mathscr Q(z_0)}} ,z_0+\frac{\mathscr U(z_0) \eta}{\sqrt{n \det\partial\bar\partial\mathscr Q(z_0)}}\right)\\
        \equiv \left(1+\mathcal O\left(\frac{\log^3n}{\sqrt n}\right)\right)
        \frac12 \exp\left(\xi\cdot\eta-\frac{|\xi|^2+|\eta|^2}2\right)
        \mathrm{erfc}\left(\sum_{k=1}^d \frac{\xi_k+\overline{\eta_k}}{\sqrt {2d}}\right)
    \end{multline*}
    uniformly for $z_0\in\partial S_{\mathscr Q}$ and $\xi,\eta\in\mathbb C^d$ with $|\xi|,|\eta|=\mathcal O(\sqrt{\log n})$. 
    \end{itemize}    
\end{theorem}

We prove Theorem \ref{thm:generalMain} in Section \ref{sec:2}. 

In the case of rotational symmetric weights we have to put the following conditions. The condition for $z\to 0$ is to assure that the droplet is simply connected, i.e., the droplet is a ball centered at the origin.

\begin{theorem} \label{thm:generalMain2}
    Suppose that $\mathscr Q:\mathbb C^d\to\mathbb R$ is rotational symmetric, both $C^2$ and strictly plurisubharmonic on $\mathbb C^d\setminus\{0\}$, and assume that $z\cdot \partial \mathscr Q(z)\to 0$ as $z\to 0$.
Then the following statements are true. 
    \begin{itemize}
    \item[(i)]
    For any $z_0\in\partial S_{\mathscr Q}$ we have as $n\to\infty$ that
    \begin{multline*} 
        \frac1{\det n \partial\bar\partial\mathscr Q(z_0)} \mathscr K_n\left(z_0+\frac{\xi}{\sqrt{n \partial\bar\partial\mathscr Q(z_0)}} \frac{z_0}{|z_0|} ,z_0+\frac{\eta}{\sqrt{n \partial\bar\partial\mathscr Q(z_0)}} \frac{z_0}{|z_0|}\right)\\
        \equiv
        \left(1+\mathcal O\left(\frac{\log^3n}{\sqrt n}\right)\right)
        \frac12 \exp\left(\xi\overline\eta-\frac{|\xi|^2+|\eta|^2}2\right)
        \mathrm{erfc}\left(\frac{\xi+\overline\eta}{\sqrt 2}\right)
    \end{multline*}
    uniformly for $z_0\in\partial S_{\mathscr Q}$ and $\xi,\eta\in\mathbb C$ with $|\xi|,|\eta|=\mathcal O(\sqrt{\log n})$.
    \item[(ii)]
    For any $z_0\in\partial S_{\mathscr Q}$ there is a unitary matrix $\mathscr U(z_0)$ such that as $n\to\infty$
    \begin{multline*} 
        \frac1{\det n \partial\bar\partial\mathscr Q(z_0)} \mathscr K_n\left(z_0+\frac{\mathscr U(z_0)\xi}{\sqrt{n \det\partial\bar\partial\mathscr Q(z_0)}} ,z_0+\frac{\mathscr U(z_0) \eta}{\sqrt{n \det\partial\bar\partial\mathscr Q(z_0)}}\right)\\
        \equiv \left(1+\mathcal O\left(\frac{\log^3n}{\sqrt n}\right)\right)
        \frac12 \exp\left(\xi\cdot\eta-\frac{|\xi|^2+|\eta|^2}2\right)
        \mathrm{erfc}\left(\sum_{k=1}^d \frac{\xi_k+\overline{\eta_k}}{\sqrt {2d}}\right)
    \end{multline*}
    uniformly for $z_0\in\partial S_{\mathscr Q}$ and $\xi,\eta\in\mathbb C^d$ with $|\xi|,|\eta|=\mathcal O(\sqrt{\log n})$. 
    \end{itemize}    
\end{theorem}

We prove Theorem \ref{thm:generalMain2} in Section \ref{sec:3}. In the rotational symmetric case we can also say something about counting statistics near the edge, which have a ``local flavour''. The interested reader may find an edge scaling limit for the variance of counting statistics in Section \ref{sec:3.2}, see Theorem \ref{thm:edgeScalingCounting}. We briefly comment on the proportionality between the number variance and the entanglement entropy in the specific case $\mathscr Q(z)=|z|^2$, which for $d=1$ was observed in \cite{LacroixMajumdarSchehr2019}.\\

One may wonder whether it is an accident that $\mathscr U(z_0)$ is a unitary matrix in Theorem \ref{thm:generalMain} and Theorem \ref{thm:generalMain2}. After all, both models (i) and (ii) exhibit a high level of symmetry. We can argue on a heuristic level that, from the viewpoint of probability theory, $\mathscr U(z_0)$ should at the very least be volume preserving. Then it has determinant $1$ and is thus invertible. Then we may equivalently write the scaling limit for $\xi=\eta$ in Conjecture \ref{con:2} as
 \begin{align*} 
        \lim_{n\to\infty} \frac1{\det n \partial\bar\partial\mathscr Q(z_0)} \mathscr K_n\left(z_0+\frac{\xi}{\sqrt{n \partial\bar\partial\mathscr Q(z_0)}}\right)
        \equiv
        \frac12 
        \mathrm{erfc}\left(\sqrt 2\mathrm{Re}(\xi \cdot \alpha(z_0))\right)
\end{align*}
for some nonzero vector $\alpha(z_0)\in\mathbb C^d$. We now prove that this vector must in fact be the outward unit normal vector $\vec n(z_0)$. (Although to argue that $\mathscr U(z_0)$ can be chosen to be unitary, we only need to show that $\alpha(z_0)$ has unit norm.) We will assume here that the convergence holds on a region where $\xi$ and $\eta$ are allowed to (mildly) grow with $n$, which is the case for $d=1$ (see, e.g., \cite{MarzoMolagOrtegaCerda2026, Charlier2025}) and there is no a priori reason to suspect that this does not hold also for $d>1$. 

\begin{proposition} \label{lem:unitVector}
Under the conditions of Conjecture 1, assume that there exists a nonzero vector $\alpha(z_0)\in \mathbb C^d$ such that 
\begin{align} \label{eq:GtoU}
        \frac1{\det n \partial\bar\partial\mathscr Q(z_0)} \mathscr K_n\left(z_0+\frac{\xi}{\sqrt{n \partial\bar\partial\mathscr Q(z_0)}}\right)
        =\frac12
        \mathrm{erfc}\left(\sqrt2 \mathrm{Re}(\xi\cdot\alpha(z_0))\right)\left(1+o(1)\right)
\end{align}
holds uniformly for $|\xi|=\mathcal O(\varepsilon_n)$, where $\varepsilon_n\to\infty$ as $n\to\infty$. Assume furthermore that \eqref{eq:con1} holds pointwise for $\xi=\eta$. Then $\alpha(z_0)=\vec n(z_0)$. 
\end{proposition}

\begin{proof}
For $\xi=\eta$ the co-cycles cancel one another, and we may replace the $\equiv$ symbol by the $=$ symbol in \eqref{eq:con1}.
Since $\partial S_{\mathscr Q}$ is assumed to be smooth, the outward unit normal vector $\vec n(z_0)$ is well-defined, and this means that the (real) Hessian of $R(z)=\check{\mathscr Q}(z)-\mathscr Q(z)$ is a rank $1$ matrix at $z_0$ (all first order derivatives of $R$ at $z_0$ are $0$ since it is $C^1$ and identically $0$ on the droplet). We then have that (with $\vec n(z_0)$ seen as in $\mathbb  R^{2d}$)
\[
\nabla^2R(z_0) = 4 \Delta R(z_0) \vec n(z_0) \vec n(z_0)^T. 
\]
In particular (with $\xi$ seen as in $\mathbb  R^{2d}$)
\begin{align} \label{eq:2ndOmostDecrease}
\lim_{n\to\infty} n R\Big(z+\frac{\xi}{\sqrt n}\Big) = -|4 \Delta R(z_0)||\vec n(z_0) \cdot \xi|^2.
\end{align}
So this expression is minimal under the constraint $|\xi|=1$ if and only if $\xi=\pm \vec n(z_0)$ (but the $+$ corresponds to the outside region).
By Berman \cite[Lemma 3.3]{Berman2009} we have
\begin{align*}
\log \frac1{\det n \partial\bar\partial \mathscr Q(z_0)} & \mathscr K_N\left(z_0+\frac{\xi}{\sqrt{n \partial\bar\partial \mathscr Q(z_0)}}\right)\\
&\leq N (\check {\mathscr Q}-\mathscr Q)\left(z_0+\frac{\xi}{\sqrt{n \partial\bar\partial \mathscr Q(z_0)}}\right)+C
\end{align*}
for some uniform constant $C>0$. To get a lower bound, we may follow the same argumentation as Berman \cite[Theorem 3.7]{Berman2009}, but with one important difference. We note that by the extremal property of the Bergman function
\begin{multline*}
\frac1{\det n \partial\bar\partial \mathscr Q(z_0)}\exp\left(N \mathscr Q\left(z_0+\frac{\xi}{\sqrt{n \partial\bar\partial \mathscr Q(z_0)}}\right)\right) 
\mathscr K_N\left(z_0+\frac{\xi}{\sqrt{n \partial\bar\partial \mathscr Q(z_0)}}\right)\\
= \max_{f_N\in \mathcal H_N\setminus\{0\}} \frac{\left|f_N(\xi)\right|^2}{\int_{\mathbb C^d} |f_N(z)|^2 e^{-N \mathscr Q_n(z)} d\omega(z)},
\end{multline*}
where $\mathcal H_N$ denotes the space of polynomials on $\mathbb C^d$ of degree $<N$, and
\[
\mathscr Q_n(\xi) = \mathscr Q\left(z_0+\frac{\xi}{\sqrt{n \partial\bar\partial \mathscr Q(z_0)}}\right).
\]
Then proceeding as Berman, following the exact same reasoning as in the proof of \cite[Theorem 3.7]{Berman2009} for the potential $\mathscr Q_n$ (rather than $\mathscr Q$), we also get a lower bound and we infer that uniformly 
\[
 (\check {\mathscr Q}-\mathscr Q)\left(z_0+\frac{\xi}{\sqrt{n \partial\bar\partial \mathscr Q(z_0)}}\right)
 = \frac1N \log \frac{\mathscr K_N\left(z_0+\frac{\xi}{\sqrt{n \partial\bar\partial \mathscr Q(z_0)}}\right)}{\det n \partial\bar\partial \mathscr Q(z_0)} 
 + \mathcal O(1/N)
\]
where the constant implied can be chosen independently from $N$ and $n$. Now let us denote $t_n=n/N$ and we will let $n,N\to\infty$ such that $t_n\sim 1/\varepsilon_n^2$ as $n\to\infty$. Under the assumptions of Conjecture \ref{con:1} we have
\[
\frac{\mathscr K_N\left(z_0+\frac{\xi}{\sqrt{n \partial\bar\partial \mathscr Q(z_0)}}\right)}{t_n^d\det n \partial\bar\partial \mathscr Q(z_0)} 
= \frac{1}2 \mathrm{erfc}(\sqrt 2 \mathrm{Re}(\xi \cdot \alpha(z_0)/\sqrt{t_n}))(1+o(1))
\]
as $n\to\infty$, uniformly for $\xi\in\mathbb C^d$ in compact sets. 
We infer that
\begin{align*}
 n (\check {\mathscr Q}-\mathscr Q)\left(z_0+\frac{\xi}{\sqrt{n \partial\bar\partial \mathscr Q(z_0)}}\right)
 = t_n \log\left(\mathrm{erfc}(\sqrt 2 \mathrm{Re}(\xi \cdot \alpha(z_0)/\sqrt{t_n}))\right)
 + o(t_n)
\end{align*}
as $n\to\infty$, uniformly for $\xi\in\mathbb C^d$ in compact sets. Using the asymptotic behavior of the erfc function for large arguments we get
\[
 \lim_{n\to\infty} n (\check {\mathscr Q}-\mathscr Q)\left(z_0+\frac{\xi}{\sqrt{n \det \partial\bar\partial \mathscr Q(z_0)}}\right)
 = - (\mathrm{Re}(\xi\cdot \alpha(z_0)))^2.
\]
This is minimal under the constraint $|\xi|=1$ if and only if $$\xi=\pm \frac{\alpha(z_0)}{|\alpha(z_0)|}$$ (by Cauchy-Schwarz applied on $\mathbb R^{2d}$). Comparing this with \eqref{eq:2ndOmostDecrease}, we infer that the outward unit normal vector is given by
\[
\vec n(z_0) = \pm \frac{\alpha(z_0)}{|\alpha(z_0)|}.
\]
Plugging this in \eqref{eq:con1}, and comparing with \eqref{eq:GtoU}, we infer that $\pm |\alpha(z_0)|=1$. 
\end{proof}

This result can be extended to $\xi\neq\eta$ by polarization, and this means that we can alternatively write the limiting behavior in Conjecture \ref{con:2} as
    \begin{multline*}
        \lim_{n\to\infty} \frac1{\det n \partial\bar\partial\mathscr Q(z_0)} \mathscr K_n\left(z_0+\frac{\xi}{\sqrt{n \partial\bar\partial\mathscr Q(z_0)}} ,z_0+\frac{\eta}{\sqrt{n \partial\bar\partial\mathscr Q(z_0)}}\right)\\
        \equiv
        \frac12 \exp\left(\xi\cdot\eta-\frac{|\xi|^2+|\eta|^2}2\right)
        \mathrm{erfc}\left(\frac{\xi\cdot \vec n(z_0)+\vec n(z_0)\cdot \eta}{\sqrt {2}}\right),
    \end{multline*}
although we prefer the universal, geometry-independent, form of the scaling limit in Conjecture \ref{con:2}. 
The interpretation of Proposition \ref{lem:unitVector}, as well as the idea behind its proof, is that the outward unit normal vector $\vec n(z_0)$ on the edge $\partial S_{\mathscr Q}$ at $z_0\in \partial S_{\mathscr Q}$ should correspond to the direction where the probability of finding points decreases the fastest.

Next, we prove a functional analytic result for the limiting kernel. The Bargmann-Fock space $\mathcal F(\mathbb C^d)$ is defined as the space of entire functions that are square-integrable with respect to the Gaussian measure, in other words
\[
\mathcal F(\mathbb C^d) = \left\{f:\mathbb C^d\to\mathbb C\text{ entire } : \|f\|_{\mathcal F}<\infty\right\},
\]
where we define the norm by
\[
\|f\|_{\mathcal F}^2 = \int_{\mathbb C^d} |f(z)|^2 e^{-|z|^2} d\omega(z).
\]
The inner product $\langle \cdot, \cdot \rangle_{\mathcal F}$ on $\mathcal F(\mathbb C^d)$ is induced by this norm. 
For any $\xi\in\mathbb C^d$ we denote $\xi^2=\xi_1^2+\ldots+\xi_d^2$. 
The multidimensional Bargmann transform \cite{Bargmann1961}, which we define explicitly as
\begin{align*}
\mathcal B[f](\xi) &= \frac1{(2\pi)^{d/4}}\int_{\mathbb R^d} f(x) e^{\xi\cdot x-\frac12\xi^2-\frac14|x|^2} dx_1\cdots dx_d,
\end{align*}
is known to act as a unitary operator from $L^2(\mathbb R^d)$ to $\mathcal F(\mathbb C^d)$.
The Hermitian-analytic part of our limiting kernel in \eqref{eq:con2} is reproducing on a specific subspace of $\mathcal F(\mathbb C^d)$. For $d=1$ the following result was proved in \cite{HaimiHedenmalm2013, AmeurKangMakarov2019}, and we generalize it to what we believe is the analogous statement for $d>1$. 

\begin{theorem} \label{thm:Functional}
Let $v\in\mathbb R^d$ be a fixed unit vector.
The kernel
\[
\frac12 \exp\left(\xi\cdot\eta\right)
        \mathrm{erfc}\left(\frac{\xi\cdot v+v\cdot \eta}{\sqrt {2}}\right)
\]
is the reproducing kernel on the subspace $\mathcal H\subset \mathcal F(\mathbb C^d)$ of functions $F$ satisfying
\[
\left|e^{\frac12\xi^2} F(\xi)\right|=\mathcal O(1) \quad \text{ uniformly for }\xi\in\mathbb C^d \text{ with }\mathrm{Re}\,\xi \in v \mathbb R_+.
\]
Furthermore, $\mathcal H$ is the isometric image of $L^2(\{x\in\mathbb R^d : x\cdot v\leq 0\})$ under the multidimensional Bargmann transform $\mathcal B$. 
\end{theorem}

\begin{proof}
Without loss of generality, by rotating coordinates, we set $v= \frac1{\sqrt d} (1, \ldots, 1)$. For the second part, we follow an argument similar to \cite{HaimiHedenmalm2013}. Let $M(\xi, \eta)=M_\eta(\xi)$ be the reproducing kernel for the space $\mathcal B[L^2(\{x\in\mathbb R^d : x\cdot v\leq 0\})]$. Then $M_\eta\in \mathcal B[L^2(\{x\in\mathbb R^d : x\cdot v\leq 0\})]$ and for any $F\in \mathcal B[L^2(\{x\in\mathbb R^d : x\cdot v\leq 0\})]$ 
\begin{align*}
\mathcal B[F](\eta) &= \langle \mathcal B[F], M_\eta\rangle_{\mathcal F}
= \langle F, \mathcal B^{-1}[M_\eta]\rangle_{L^2(\mathbb R^d)},
\end{align*}
where we used that the Bargmann transform is a unitary operator. Since $F$ was arbitrary, it follows using the definition of $\mathcal B$ that
\begin{align*}
\mathcal B^{-1}[M_\eta](x) = \frac1{(2\pi)^{d/4}}\mathfrak{1}_{x\cdot v\leq 0}(x) e^{\overline\eta\cdot x-\frac12 \overline \eta^2-\frac14 |x|^2}. 
\end{align*}
Inverting this equation, we get
\begin{align*}
M_\eta(\xi) &= \mathcal B[x\mapsto \frac1{(2\pi)^{d/4}}\mathfrak{1}_{x\cdot v\geq 0}(x) e^{\overline\eta\cdot x-\frac12 \overline \eta^2-\frac14 |x|^2}](\xi)\\
&= \frac1{(2\pi)^{d/2}} e^{\xi\cdot\eta} \int_{\mathbb R^{d-1}} \int_{-\infty}^{-\sum_{k=2}^{d} x_k} e^{-\frac12\sum_{k=1}^d (x_k-\xi_k-\overline\eta_k)^2} dx_1\cdots dx_d\\
&= \frac12 e^{\xi\cdot \eta} \mathrm{erfc}\left(\frac{\xi\cdot v+v\cdot\eta}{\sqrt 2}\right),
\end{align*}
where the last step follows from Lemma \ref{lem:GaussianErfc} in Appendix \ref{sec:A}. Hence our kernel is indeed the reproducing kernel on $\mathcal B(L^2(\{x\in\mathbb R^d : x\cdot v\leq 0\}))$.

We shall now prove that $\mathcal B(L^2(\{x\in\mathbb R^d : x\cdot v\leq 0\}))=\mathcal H$. Suppose that $F=\mathcal B[f]$, where $f\in L^2(\{x\in\mathbb R^d : x\cdot v\leq 0\})$. 
Then, again using Lemma \ref{lem:GaussianErfc}, combined with Cauchy-Schwarz
\begin{align*}
|F(\xi) e^{\frac12\xi^2}| &= \left|\frac1{(2\pi)^{d/4}} \int_{x\cdot v\leq 0} f(x) e^{\xi\cdot x-\frac14|x|^2} dx_1\cdots dx_d\right|\\
&\leq \frac{1}{\sqrt 2} \|f\|_{L^2(\mathbb R^d)} e^{|\mathrm{Re} \, \xi|^2} \left|\mathrm{erfc}\left(\mathrm{Re}(\xi) \cdot v\right)\right|, 
\end{align*}
which is uniformly bounded for $\mathrm{Re}\,\xi\in v \mathbb R_+$. Now let us assume that $F\in\mathcal H$. We shall prove that $F\in \mathcal B(L^2(\{x\in\mathbb R^d : x\cdot v\leq 0\}))$. 
We can write
\begin{align*}
G(\xi) := e^{-\frac12\overline \xi^2} F(-i\overline \xi) &= \frac1{(2\pi)^{d/4}}\int_{\mathbb R^d} e^{-i \overline \xi\cdot x} f(x) e^{-\frac14|x|^2} dx_1\cdots dx_d,
\end{align*}
where $f=\mathcal B^{-1} F\in L^2(\mathbb R^d)$. We have to argue that $f$ has support in the set of $x\in\mathbb R^d$ such that $x\cdot v\leq 0$, which we shall prove with a Paley-Wiener type theorem from \cite{LiDeQi2018}. 
We denote the cone $\Gamma=v \mathbb R_+$ and the tube $T_\Gamma=\{\xi\in\mathbb C^d : \mathrm{Im} \, \xi \in \Gamma\}$. The conditions on $F$ then imply that $H(\xi) := e^{-\frac12|\xi|^2} G(\xi)$ is in the Hardy space $H^\infty(T_\Gamma)$. Then \cite[Theorem 2.2]{LiDeQi2018} implies that the Fourier transform $\hat H$ has support in the dual cone $\Gamma^*=\{x\in \mathbb R^d : x\cdot y\geq 0 \,\, \forall y\in \Gamma\}=\{x\in \mathbb R^d : x\cdot v\geq 0\}$. That $F\in \mathcal B(L^2(\{x\in\mathbb R^d : x\cdot v\leq 0\}))$ now follows from the identity
\begin{align*}
\hat H(-x) &= \frac1{(2\pi)^{d/4}}  \int_{\mathbb C^d} e^{i \xi \cdot x} H(\xi) d\omega(\xi)\\
&= \frac1{(2\pi)^{d/4}} \int_{\mathbb C^d} e^{\overline\xi\cdot x} e^{-\frac12\overline\xi^2-\frac12|x|^2-\frac12|\xi|^2} F(\xi) d\omega(\xi)
= (\mathcal B^{-1} F)(x).
\end{align*}
\end{proof}

Analogous statements can be extracted from Theorem \ref{thm:Functional} when $|v|\neq 1$ or $v\in\mathbb C^d$, by applying suitable coordinate transformations. As for the $d=1$ case, for $d>1$ one may also interpret our model as a quantum Hall fluid (with the wave function constructed with the Slater determinant of our weighted polynomials). Theorem \ref{thm:Functional} implies that the Lowest Landau Levels that describe the edge correspond precisely to those functions in the Bargmann-Fock space that may have support near the edge inside the fluid (the droplet), but quickly decay in the outward normal direction.

Finally, we note that for $d>1$ there is an interesting feature where regular edge points $z_0\in\partial S_{\mathscr Q}$ may exhibit a certain type of bulk degeneracy: one or more coordinates of $z_0$ could arise as a limiting bulk point. This is perhaps best illustrated by the pluripotential version of the Ginibre ensemble, 
\[
\mathscr Q(z)=|z|^2=Q_1(z_1)+\ldots+Q_d(z_d)
\] 
with planar potentials $Q_k(z)=|z|^2$. Then $\partial S_{\mathscr Q}$ is the unit sphere in $\mathbb C^d$ which contains a point such as $z_0=(\zeta_0,0)$ where $\zeta_0$ lies on the unit sphere in $\mathbb C^{d-1}$. Then the last coordinate is, in a sense, the deepest point in the bulk, the unit disk, associated to the potential of the last coordinate $z_d=0$, where $Q_d$ attains its minimum. We will explain this situation in greater generality in Section \ref{sec:2}. As it turns out, to prove the edge scaling limits, this requires us to understand the planar kernels for each such coordinates where the number of terms in the sum defining the kernel is not $n$, but grows slower than $n$. To the best of our knowledge, this setting has not yet been explored in the literature. Since we believe this result, as well as our method of proof, is of independent interest, we state it here in the current section. 

\begin{theorem} \label{thm:PartialBerg}
Let $Q:\mathbb C\to\mathbb R$ be a real-analytic function satisfying \eqref{eq:growthCond}, with a unique minimum at $z=0$. Assume that $m_n$ is a sequence of natural numbers converging to $\infty$. Then there exists a constant $r_Q>0$ such that
\[
\frac{e^{-n Q\big(z/\sqrt{n \Delta Q(0)}\big)}}{n\Delta Q(0)}  \sum_{j=0}^{m_n} \left|P_j\left(\frac{z}{\sqrt{n\Delta Q(0)}}\right)\right|^2 = 1+\mathcal O\left(\frac{m_n}{n}\right)
\]
as $n\to\infty$, uniformly for all $|z|\leq r_Q \sqrt{m_n}$. If we also have $m_n=o(n^{2/3})$, then
\begin{align*}
\lim_{n\to\infty} &\frac{e^{-\frac12n Q\big(\frac{\sqrt{m_n} z_0+\xi}{\sqrt{n\Delta Q(0)}}\big)} e^{-\frac12n Q\big(\frac{\sqrt{m_n} z_0+\eta}{\sqrt{n\Delta Q(0)}}\big)}}{n\Delta Q(0)}\\
&\qquad\sum_{j=0}^{m_n} P_j\left(\frac{\sqrt{m_n} z_0+\xi}{\sqrt{n\Delta Q(0)}}\right) 
\overline{P_j\left(\frac{\sqrt{m_n} z_0+\eta}{\sqrt{n\Delta Q(0)}}\right)}\equiv e^{\xi \overline \eta-\frac12(|\xi|^2+|\eta|^2)}
\end{align*}
as $n\to\infty$, uniformly for $|z_0|\leq r_Q$ and $\xi,\eta\in\mathbb C$ in compact sets.
\end{theorem}

One way to prove such results, as we have been able to verify, is with a well-known approach involving Hörmander's $\bar\partial$-method \cite{Hormander1965}. However, we devise a method that eventually allows one to approximate the kernel using the Lagrange multiplier method. In particular, our method gives an approximation uniformly on $\mathbb C$, see Proposition \ref{prop:planarKernelApprox}, while Hörmander's $\bar\partial$-method typically yields approximations locally. 
With some effort, one can reduce the regularity conditions to $Q$ being $C^4$. This can be proved by Taylor expanding $Q$ and neglecting terms beyond fourth order. 

\subsection*{Outlook}

Finally, we comment on how our results may be extended. To fully prove Conjecture \ref{con:1} and Conjecture \ref{con:2}, one probably has to invent a new method. For $d=1$, there are essentially three general approaches. For $d=1$ the local edge universality was first proved by Hedenmalm and Wennman in \cite{HedenmalmWennman2021} using approximately
orthogonal quasipolynomials, constructed using an orthogonal foliation flow. Later, Hedenmalm published a related approach, using so-called soft Riemann-Hilbert problems \cite{Hedenmalm2024}, starting from a viewpoint first set out by Its and Takhtajan \cite{ItsTakhtajan2007}. Then there is also the recent paper by Wennman and Cronvall \cite{CronvallWennman2025}. The method starts with the extremal property of the Bergman kernel (on the diagonal) on the space $\mathcal H$ in Theorem \ref{thm:Functional} above for $d=1$. Then  they construct peak polynomials to get a lower bound for the rescaled polynomial Bergman kernel. All three approaches seem to suffer from the same drawback for $d>1$, namely that they rely heavily on the fact that there is a conformal map from the exterior of the droplet to the exterior of the closed unit disk, and this is used in a nonlocal way. For $d>1$ such a map does not exist in general. One other obstruction is that for $d>1$, contrary to the $d=1$ case, the obstacle function $\check{\mathscr Q}$ is in general not harmonic outside the droplet, an ingredient that is used in the construction of the peak polynomials in \cite{CronvallWennman2025}. Nevertheless, the approach in \cite{CronvallWennman2025} appears robust, and armed with our Theorem \ref{thm:Functional} there is some hope that one may prove Conjecture \ref{con:1} and Conjecture \ref{con:2}. 

In a different direction, there are also more exotic settings to be explored. For example, one may consider situations with a \textit{hard edge}, where the value of $\mathscr Q$ suddenly becomes $+\infty$ and particles are excluded from a certain region (the setting of the current paper, where the edge is determined by a smoothly behaved potential, is that of a \textit{soft edge}). For $d=1$, this was considered in \cite{AmeurKangMakarov2019, Seo2022, AmeurCharlierCronvall2024} and finally proved in generality in \cite{CronvallWennman2025}. Another interesting setting is that of \textit{singular boundary points}, for $d=1$ considered, e.g., in \cite{AmeurKangMakarovWennman2020}. We are already investigating an explicit model with singular boundary points for $d>1$, and hope to publish our results in the near future.

%
%
%

\section{A factorization into planar weights} \label{sec:2}

In this section we will prove Theorem \ref{thm:generalMain}. 
Henceforth, we assume that
\begin{align*}
    \mathscr Q(z) = \sum_{k=1}^d Q_k(z_k).
\end{align*}
We will assume that each $Q_k$ is $C^2$ and satisfies the growth condition
\begin{align} \label{eq:growthQ}
\liminf_{|z|\to\infty} \frac{Q_k(z)}{\log |z|^2} > 1.
\end{align}

\subsection{Preparation: some planar potential theory}

For any $\tau\geq 0$, and any $Q:\mathbb C\to\mathbb R$ satisfying the above growth condition, we define the $\tau$-obstacle function $\check Q_{\tau} : \mathbb C\to[0,\infty)$ as the maximal subharmonic function $q$ such that $q\leq Q$ and
\[
q(z) \leq \tau \log |z|^2 + \mathcal O(1)
\]
as $|z|\to \infty$. In fact, the solution satisfies 
\begin{align} \label{eq:growthCheckQ}
\check Q_\tau(z)=\tau\log|z|^2+\mathcal O(1)
\end{align}
as $|z|\to\infty$. We define the $\tau$-predroplet as the coincidence set
\[
S_{Q,\tau}^\star = \{z\in\mathbb C : \check Q_{\tau}(z)=Q(z)\}.
\]
For $\tau>0$ (and $d=1$), the $\tau$-droplet $S_{Q,\tau}$ is defined as the support of the unique minimizer of the functional
\begin{align} \label{eq:MinProb}
J(\mu) = \int_{\mathbb C} \int_{\mathbb C} \log \frac1{|z_j-z_k|} \, d\mu(z) d\mu(w) + \int_{\mathbb C} Q(z) d\mu(z)
\end{align}
over all compactly supported Borel  measures $\mu$ on $\mathbb C$ with total mass $\tau$, while for $\tau=0$ we define $S_{Q,0}$ as the set of $z\in\mathbb C$ where $Q_k$ attains its minima. Note that we automatically have $\check Q_{0}=\min Q$ and $S_{Q,0}=S_{Q,0}^\star$. Note that by the maximality of $\check Q_\tau$, we have
\begin{align} \label{eq:SQtauSeq}
S_{Q,\tau}^\star \subset S_{Q,\tau'}, \quad 0\leq \tau\leq \tau'.
\end{align}
Further down in this section we shall consider several potentials $Q_k$ with $k=1,\ldots, d$ and then we denote the corresponding expressions as $\check Q_{k,\tau}, S_{Q_k,\tau}$ and $S_{Q_k,\tau}^\star$.

We repeat a definition that was used in \cite{HedenmalmWennman2021}.

\begin{definition}[$\tau$-admissibility] \label{def:tauAdm}
Let $\tau> 0$. 
We say that $Q:\mathbb C\to\mathbb R$ is $\tau$-admissible if $S_{Q,\tau}=S_{Q,\tau}^\star$ and all of the following are satisfied:
\begin{itemize}
\item[(i)] $Q$ is $C^2$.
\item[(ii)] $Q$ is real-analytic and strictly subharmonic in a neighborhood of $S_{Q,\tau}$. 
\item[(iii)] $Q$ grows sufficiently fast at infinity:
\[
\liminf_{|z|\to\infty} \frac{Q(z)}{\log |z|^2} > \tau.
\]
\item[(iv)] $\partial S_{Q,\tau}$ is a smooth Jordan curve.
\end{itemize}
\end{definition}
The last condition in particular implies that the $\tau$-droplet is simply connected.
If $Q$ is $\tau_0$-admissible, the conditions imply that $\partial S_{\tau}$ is real-analytically smooth in a neighborhood of $\tau=\tau_0$, as proved in \cite{HedenmalmShimorin2002} with the help of Sakai's work \cite{Sakai1991}. 
We now extend this definition to hold for a range of $\tau$. 

\begin{definition} \label{def:01Adm}
We say that $Q:\mathbb C\to\mathbb R$ is $[0,1]$-admissible if it is $\tau$-admissible for all $\tau\in (0,1+\delta)$ for some $\delta>0$, and furthermore that $S_{Q,0}=\{p_Q\}$ for some $p_Q\in\mathbb C$ (equivalently, that $S_{Q,0}$ is connected).
\end{definition}

The second condition, that $S_{Q,0}$ consists of a single element, is added to assure that $S_{Q,\tau}$ is simply connected for any $\tau$, including $\tau=0$. There are examples of potentials where a topological change occurs as $\tau$ varies, e.g., see \cite{BaloghMerzi2015, Byun2024}. On a heuristic level one could imagine examples of potentials $Q$ such that $S_{Q,\tau}$ is connected for all $\tau\in (0,1]$, but where a topological change occurs at $\tau=0$, and $S_{Q,0}$ consists of more than one element. 

By translation we may always assume without loss of generality that $p_Q=0$. 

\begin{lemma} \label{lem:CheckQtauTo}
If $Q:\mathbb C\to\mathbb R$ is $[0,1]$-admissible then pointwise
\begin{align} \label{eq:limCheckQtau0}
\lim_{\tau\to 0^+} \check Q_\tau = Q_0 = \min Q. 
\end{align}
\end{lemma}

\begin{proof}
For each fixed $z\in\mathbb C$, $\tau\mapsto \check Q_\tau(z)$ is an increasing function of $\tau$ satisfying the lower bound $\check Q_\tau(z)\geq \min Q$. Hence we are guaranteed that the limit in \eqref{eq:limCheckQtau0} exists, let us denote it by $\check Q_*$. Next, we should argue that it equals $\min Q$. For $z=p_Q$, by \eqref{eq:SQtauSeq}, we find trivially
\[
\check Q_*(p_Q)=\lim_{\tau\to 0^+} \check Q_\tau(p_Q) = \lim_{\tau\to 0^+} \check Q_0(p_Q) = \min Q.
\]
For any $z\neq p_Q$ we have $z\in \mathbb C\setminus S_{Q,\tau}$ for $\tau$ small enough. It is a well-known fact that $\check Q_\tau$ is harmonic outside its $\tau$-droplet for any $\tau>0$. Hence, in some bounded neighborhood of $z$, a decreasing sequence $(\check Q_{\tau_k})_k$ which is bounded from below may be constructed, where $\tau_k$ is strictly decreasing with limit $0$. By Harnack's principle \cite{Kellogg1967}, this implies that our sequence convergences to a harmonic function, uniformly on our neighborhood. We conclude that $\check Q_*$ is harmonic on $\mathbb C\setminus\{p_Q\}$. Then $\check Q_*$, restricted to $\mathbb C\setminus\{p_Q\}$ has a removable singularity at $p_Q$. 
We can construct a (possibly different) decreasing sequence of positive $\tau_k$ converging to $0$, and a sequence $p_k\in \partial S_{Q,\tau_k}\setminus\{p_Q\}$ such that $p_k\to p_Q$ as $k\to\infty$. If this were not possible, then, due to \eqref{eq:SQtauSeq}, there would exist an $\varepsilon>0$ such that for $\tau'>0$ small enough
\begin{align*}
\{z\in\mathbb C : |z-p_Q|<\varepsilon\} \subset \bigcap_{0\leq \tau<\tau'} S_{Q, \tau} \subset S_{Q,0} = \{p_Q\},
\end{align*}
a contradiction. Since $S_{Q,\tau}^\star=S_{Q,\tau}$, we have $\check Q_\tau=Q$ on $\partial S_{Q,\tau}$ and the continuity of $Q$ yields
\[
\min Q\leq \lim_{k\to\infty} \check Q_*(p_k)
\leq \lim_{k\to\infty} \check Q_{\tau_k}(p_k)
= \lim_{k\to\infty} Q(p_k) = Q(p_Q) = \min Q.
\]
Hence the value dictated by the removable singularity coincides with $Q_*(p_Q)=\min Q$, and we conclude that $\check Q_*$ is a harmonic function on $\mathbb C$. Since, for any fixed $z\in\mathbb C$, $\tau\mapsto \check Q_\tau(z)$ is decreasing, we have, e.g., $\check Q_*\leq \check Q_1$. We infer that $\check Q_*$ is a harmonic function on $\mathbb C$ satisfying the growth condition
\[
\check Q_*(z) \leq \log |z|^2+\mathcal O(1), \quad |z|\to \infty. 
\]
Since this is slower than linear growth, a version of Liouville's theorem tells us that $\check Q_*$ is constant, and thus $\check Q_*=\min Q$ identically. We have proved \eqref{eq:limCheckQtau0} as a pointwise limit.
\end{proof}

With an argument involving the Herglotz transform (see, e.g., \cite{MarzoMolagOrtegaCerda2026}) one may argue that $\check Q_\tau(z)$ is a real-analytic function of $\tau$ on $(0,1]$ when it is $[0,1]$-admissible. Combined with Lemma \ref{lem:CheckQtauTo} this yields the following corollary. 

\begin{corollary} \label{cor:QtauCont01}
If $Q:\mathbb C\to\mathbb R$ is $[0,1]$-admissible then for each $z\in\mathbb C$ the function $\tau\mapsto \check Q_\tau(z)$ is continuous on $[0,1]$. 
\end{corollary}

\subsection{Preparation: some pluripotential theory}

Let us now return to the higher-dimensional weight with potential 
\begin{align*}
\mathscr Q(z)=Q_1(z_1)+\ldots+Q_d(z_d), \qquad z\in\mathbb C^d.
\end{align*}
In our particular setting, the Monge-Ampère measure of $\mathscr Q$, due to its specific decomposition, is explicitly given by
\begin{align*}
    \mathfrak{1}_{S_{\mathscr Q}}(z) \partial\bar\partial \mathscr Q(z) d\omega(z) = \mathfrak{1}_{S_{\mathscr Q}}(z) \prod_{k=1}^d \Delta Q_k(z_k) dA(z_k), \qquad z\in\mathbb C^d.
\end{align*}

\begin{proposition} \label{thm:obstacleFunction}
    Suppose that each $Q_k$ is $[0,1]$-admissible. Then the obstacle function as defined in \eqref{eq:defPluriObstF} is explicitly given by
    \begin{align} \label{eq:pluriObstacle}
    \check{\mathscr Q}(z) = \max_{\substack{\tau_1 \, , \, \ldots \, , \, \tau_d\geq 0\\
        \tau_1+\ldots+\tau_d=1}}
        \sum_{k=1}^d \check{Q}_{k,\tau_k}(z_k).
    \end{align}
\end{proposition}

\begin{proof}
    We have for all $z\in\mathbb C^d$ that $\check Q_{k,\tau}(z_k)\leq Q_k(z_k)$ for any $\tau\in (0,1]$ and $k=1,\ldots,d$. Thus $\check{\mathscr Q}$ as defined in \eqref{eq:pluriObstacle} satisfies
    \[
    \check{\mathscr Q}(z)\leq 
    \max_{\substack{\tau_1 \, , \, \ldots \, , \, \tau_d\geq 0\\
        \tau_1+\ldots+\tau_d=1}}
        \sum_{k=1}^d Q_k(z_k) = \mathscr Q(z).
    \]
    Furthermore, as $|z|\to\infty$, we have
    \begin{align*}
    \check{\mathscr Q}(z)
        \leq \max_{\substack{\tau_1 \, , \, \ldots \, , \, \tau_d\geq 0\\
        \tau_1+\ldots+\tau_d=1}}
        \sum_{k=1}^d \tau_k \log|z|^2+\mathcal O(1)
        = \log |z|^2+\mathcal O(1).
    \end{align*}
    Thus $\check{\mathscr Q}$ satisfies the required properties, except for the maximality, which we now prove. We follow a proof style similar to Klimek \cite{Klimek1991}. Let $q\in\mathcal L(\mathbb C^d)$ such that $q\leq \mathscr Q$. Then the functions where we fix all but one variables to be some $p\in\mathbb C$ are subharmonic functions of at most logarithmic growth. For example
    \begin{align*}
        q(z_1,p,\ldots,p) \leq \log|(z_1,p,\ldots,p)|^2+\mathcal O(1)\leq \log |z_1|^2+\mathcal O(1)
    \end{align*}
    as $|z_1|\to\infty$. Now define
    \begin{align*}
        \tau_1^* = \limsup_{R\to\infty} \frac{\displaystyle\sup_{|z_1|=R}q(z_1,p,\ldots,p)}{\log R^2},
    \end{align*}
    and similarly $\tau_k^*$ for the $d-1$ other functions. Then by maximality we must have
    \begin{align*}
        q(z_1,p,\ldots,p)- \sum_{k=2}^d Q_k(p)\leq \check Q_{1,\tau_1^*}(z)
    \end{align*}
    and similarly for $k=2,\ldots,d$. Notice that for fixed $z_2$
    \begin{align*}
        q(z_1,z_2,p,\ldots,p)-Q_2(z_2)-\sum_{k=3}^d Q_k(p)
    \end{align*}
    defines a subharmonic function on $\mathbb C$ which is $\leq Q(z_1)$ and satisfies the growth condition
    \[
    q(z_1,z_2,p,\ldots,p)-Q_2(z_2)- \sum_{k=3}^d Q_k(p)\leq\tau_1^*\log|z_1|^2+\mathcal O(1)
    \]
    as $|z_1|\to\infty$. Thus, by maximality
    \[
    q(z_1,z_2,p,\ldots,p)\leq \check Q_{1,\tau_1^*}(z_1)+Q_2(z_2)+ \sum_{k=3}^d Q_k(p).
    \]
    By symmetry we get a similar inequality for $z_2$ and thus
    \begin{align*}
    q(z_1,z_2,p,\ldots,p)
    &\leq \frac12 Q_1(z_1)+\frac12\check Q_{1,\tau_1^*}(z_1)+\frac12 Q_2(z_2)+\frac12 \check Q_{2,\tau_2^*}(z_2)+ \sum_{k=3}^d Q_k(p)\\
    &\leq \check Q_{1,\tau_1^*}(z_1)+\check Q_{2,\tau_2^*}(z_2)+ \sum_{k=3}^d Q_k(p).
    \end{align*}
    This argument may be repeated by induction and we obtain
    \[
    q(z)\leq  \sum_{k=1}^d \check Q_{k,\tau_k^*}(z_k).
    \]
    If $\tau_1^*+\ldots+\tau_d^*<1$, then we may simply increase some of the $\tau_k^*$ until the sum is $1$. This will give us a function that dominates $q$ and is still $\leq \mathscr Q$ while being $\leq \log|z|^2+\mathcal O(1)$ as $|z|\to \infty$. However, that function in turn is dominated by $\check{\mathscr Q}$ as defined in \eqref{eq:pluriObstacle}. 
\end{proof}

\begin{lemma} \label{lem:MAMeas}
   Assume that each $Q_k$ is $[0,1]$-admissible. Then the Monge-Ampère measure is given by 
    \[
    \det\partial\bar\partial \mathscr Q(z) \, d\omega(z) = \max_{\substack{\tau_1 \, , \, \ldots \, , \, \tau_d\geq 0\\
        \tau_1+\ldots+\tau_d=1}}
        \prod_{k=1}^d \Delta \check Q_{k,\tau_k}(z_k) \, d\omega(z).
    \]
\end{lemma}

\begin{proof}
Consider a sequence $\check{\mathscr Q}_n$ given explicitly by
\begin{align} \label{eq:discreteCheckQ}
\check{\mathscr Q}_n(z) = \max_{\substack{j_1 \, , \, \ldots, j_d\in\mathbb N_0\\ j_1+\ldots+j_d=n}} 
\sum_{k=1}^d \check Q_{k,j_k/n}(z_k).
\end{align}
The pointwise maximum of a finite number of plurisubharmonic functions is again plurisubharmonic. 
Hence $\check{\mathscr Q}_n$ is a sequence of increasing locally bounded plurisubharmonic functions. 
Furthermore, we have the bounds
\[
\check{\mathscr Q}_n(z) \leq \check{\mathscr Q}(z)\leq \mathscr Q(z).
\]
Thus, for any $z\in\mathbb C^d$, $\check{\mathscr Q}_n(z)$ converges as $n\to\infty$. We will show that it in fact converges to $\check{\mathscr Q}(z)$. Fix a $z\in\mathbb C$. For each $n$, we find a multi-index $j^{(n)}(z)=(j_1^{(n)}(z), \ldots, j_d^{(n)}(z))\in[0,n]^d$ which yields the maximum on the right-hand side in \eqref{eq:discreteCheckQ}. By possibly taking a subsequence, we may assume that
\[
\lim_{n\to\infty} \frac{j^{(n)}(z)}{n} = \tau(z) = (\tau_1(z), \ldots, \tau_d(z)),
\]
for some limit $\tau(z)\in [0,1]^d$. Then by Corollary \ref{cor:QtauCont01} we have
\[
\lim_{n\to\infty} \check{\mathscr Q}_n(z) = \sum_{k=1}^d \lim_{\tau_k\to \tau_k(z)} \check Q_{k, \tau_k}(z_k) 
= \sum_{k=1}^d \check Q_{k, \tau_k(z)}(z_k)
= \check{\mathscr Q}(z),
\]
where the last step follows by a denseness argument and \eqref{eq:pluriObstacle}.

Then by the Bedford-Taylor theorem \cite[Theorem 2.1]{BedfordTaylor1982}
\begin{align*}
   \det\partial\bar\partial \mathscr Q(z) = \lim_{n\to\infty} \det\partial\bar\partial \mathscr Q_n(z).
\end{align*}
The pointwise maximum of a finite number of $C^{1,1}$ functions is again $C^{1,1}$, hence $\check{\mathscr Q}_n$ is $C^{1,1}$ and (by Rademacher's theorem) twice differentiable $L^2(\mathbb C^d)$-a.e., and we may thus apply the Monge-Ampère operator to $\check{\mathscr Q}_n$ to get a density function
\[
\lim_{n\to\infty} \det\partial\bar\partial \mathscr Q_n(z)
= \max_{\substack{j_1 \, , \, \ldots, j_d\in\mathbb N_0\\ j_1+\ldots+j_d=n}} 
\prod_{k=1}^d \Delta\check Q_{k,j_k/n}(z_k).
\]
Obviously, we have
\[
\max_{\substack{j_1 \, , \, \ldots, j_d\in\mathbb N_0\\ j_1+\ldots+j_d=n}} 
\prod_{k=1}^d \Delta\check Q_{j_k/n}(z)
\leq \max_{\substack{\tau_1 \, , \, \ldots \, , \, \tau_d\geq 0\\
        \tau_1+\ldots+\tau_d=1}}
        \prod_{k=1}^d \Delta \check Q_{\tau_k}(z_k).
\]
Suppose the left-hand side does not converge to the right-hand side. Then there exists an $\varepsilon>0$ and a subsequence $\check{\mathscr Q}_{n_m}$ with
\[
\max_{\substack{j_1 \, , \, \ldots, j_d\in\mathbb N_0\\ j_1+\ldots+j_d=n_m}} 
\prod_{k=1}^d \Delta\check Q_{j_k/n_m}(z)
\leq -\varepsilon+ \max_{\substack{\tau_1 \, , \, \ldots \, , \, \tau_d\geq 0\\
        \tau_1+\ldots+\tau_d=1}}
        \prod_{k=1}^d \Delta Q_{\tau_k}(z_k).
\]
However, any combination $(\tau_1,\ldots,\tau_d)$ can be approximated by $(j_1,\ldots,j_d)/n_m$ if we pick $n_m$ large enough. This, in combination with continuity following from Corollary \ref{cor:QtauCont01}, yields a contradiction, and the lemma follows.
\end{proof}

\begin{proposition}
    Suppose that $Q$ is $[0,1]$-admissible. Then we have
    \begin{align*}
        S_{\mathscr Q} = \bigcup_{\substack{\tau_1 \, , \, \ldots \, , \, \tau_d\geq 0\\
        \tau_1+\ldots+\tau_d=1}}
        \prod_{k=1}^d S_{Q,\tau_k}.
    \end{align*}
\end{proposition}

\begin{proof}
    We need to find the coincidence set $\check{\mathscr Q}=\mathscr Q$. Suppose that $z\in\mathbb C$ is in the coincidence set. Thus by Proposition \ref{thm:obstacleFunction} there exists $(\tau_1,\ldots,\tau_d)\in[0,1]^d$ such that $\tau_1+\ldots+\tau_d=1$ and
\[
\sum_{k=1}^d \check Q_{k,\tau_k}(z) = \sum_{k=1}^d Q_k(z_k).
\]
Since, by definition $\check Q_{k, \tau_k}\leq Q_k$ for all $k=1,\ldots,d$,  we necessarily have
    \[
    \check Q_{k, \tau_k}(z_k)=Q_k(z_k)
    \]
    for all $k=1,\ldots,d$. Since we assume that $S_{Q,\tau_k}=S_{Q,\tau_k}^\star$, this means that
    \[
    z_k\in S_{Q,\tau_k}
    \]
    for all $k=1,\ldots,d$. We conclude that
    \begin{align*}
        S_{\mathscr Q}^\star \subset \bigcup_{\substack{\tau_1 \, , \, \ldots \, , \, \tau_d\geq 0\\
        \tau_1+\ldots+\tau_d=1}}
        \prod_{k=1}^d S_{Q,\tau_k}.
    \end{align*}
    Now suppose that $z\in\mathbb C^d$ is not in the coincidence set. This means for any $(\tau_1,\ldots,\tau_d)\in[0,1]^d$ with $\tau_1+\ldots+\tau_d=1$ that
    \begin{align*}
        \sum_{k=1}^d \check Q_{\tau_k}(z_k) \leq \check{\mathscr Q}(z) < \mathscr Q(z)
        = \sum_{k=1}^d Q(z_k).
    \end{align*}
    This means that there is at least one $\tau_k$ in any such combination such that $z_k\in \mathbb C\setminus S_{Q,\tau_k}$. We conclude that 
    \begin{align*}
        z \not\in \bigcup_{\substack{\tau_1 \, , \, \ldots \, , \, \tau_d\geq 0\\
        \tau_1+\ldots+\tau_d=1}}
        \prod_{k=1}^d S_{Q_{\tau_k}}.
    \end{align*}
    Thus it follows that
    \[
    S_{\mathscr Q}^\star = \bigcup_{\substack{\tau_1 \, , \, \ldots \, , \, \tau_d\geq 0\\
        \tau_1+\ldots+\tau_d=1}}
        \prod_{k=1}^d S_{Q_{\tau_k}}
    \]
    The droplet $S_{\mathscr Q}$ is defined as the support of the measure
    \begin{align*}
    \mathfrak{1}_{S_{\mathscr Q}^\star}(z) \prod_{k=1}^d \Delta Q(z_k) d\omega(z) &= \max_{\substack{\tau_1 \, , \, \ldots \, , \, \tau_d\geq 0\\
        \tau_1+\ldots+\tau_d=1}}
        \prod_{k=1}^d \Delta \check Q_{\tau_k}(z_k) \, d\omega(z)\\
    &= \max_{\substack{\tau_1 \, , \, \ldots \, , \, \tau_d\geq 0\\
        \tau_1+\ldots+\tau_d=1}}
        \mathfrak{1}_{S_{Q_{\tau_1}}\times \cdots \times S_{Q_{\tau_d}}}(z)
        \prod_{k=1}^d \Delta Q(z_k) \, d\omega(z),
    \end{align*}
    which, since all $Q_{\tau_k}$ are strictly subharmonic in a neighborhood of $S_{\tau_k}$, except on regions with Lebesgue measure $0$ (were one or more $\tau_k$ may be $0$), means that $S_{\mathscr Q}$ is as stated.
\end{proof}

Note that we in particular infer that $S_{\mathscr Q}=S_{\mathscr Q}^\star$ in our setting. The $[0,1]$-admissibility of the $Q_k$ implies that the droplet of $\mathscr Q$ equals the predroplet. Our next task is to describe the topological boundary of $S_{\mathscr Q}$. \\

\begin{proposition} \label{thm:dropletBoundary}
    When $Q$ is $[0,1]$-admissible we have
    \[
    \partial S_{\mathscr Q} = \,
\bigcup_{\substack{\tau_1, \, \ldots \, ,\tau_d \ge 0 \\ \tau_1+\cdots+\tau_d= 1}}
\;
\prod_{k=1}^d \partial S_{Q,\tau_k} \, .
    \]
\end{proposition}

\begin{proof}
Let $z\in \partial S_{\mathscr Q}$. Since $S_{\mathscr Q}$ is closed, this implies that $z\in S_{\mathscr Q}$. Hence there exists $\tau\in[0,1]^d$ such that $\tau_1+\ldots+\tau_d=1$ and $z_k\in S_{Q,\tau_k}$ for all $k=1,\ldots,d$. In fact, since $\partial S_{Q_k,\tau_k}$ depends real-analytically smooth on $\tau_k$, we may assume that there exists a (possibly different) $\tau\in[0,1]^d$ such that $\tau_1+\ldots+\tau_d\leq 1$ and $z_k\in \partial S_{Q,\tau_k}$ for all $k=1,\ldots,d$. Suppose that $\tau_1+\ldots+\tau_d<1$. In that case, we may find $\tau^*\in[0,1]^d$ such that $\tau_k^*\geq \tau_k$ while $\tau_1^*+\ldots+\tau_d^*=1$ and $z_k\in \mathring S_{Q,\tau_k^*}$ for all $k=1,\ldots,d$.
Clearly then, we may find an open set containing $z$ that is contained in $S_{\mathscr Q}$. This implies that $z$ is not a boundary point and we have reached a contradiction. We conclude that we must have had $\tau_1+\ldots+\tau_d=1$ from the beginning. We conclude that
\[
\partial S_{\mathscr Q} \subset
    \bigcup_{\substack{\tau_1, \, \ldots \, ,\tau_d \ge 0 \\ \tau_1+\cdots+\tau_d= 1}}
\;
\prod_{k=1}^d \partial S_{Q,\tau_k} \,.
    \]
    Now consider any point $z$ that is not a boundary point. Since $S_{\mathscr Q}$ is closed, we may assume that $z$ is in the interior or exterior of $S_{\mathscr Q}$. If $z\in \mathring S_{\mathscr Q}$, we may find $\tau\in[0,1]^d$ such that $\tau_1+\ldots+\tau_d=1$ and
    \begin{align*}
        B(z_k,\delta)\subset \mathring S_{Q,\tau_k}
    \end{align*}
    for all $k=1,\ldots,d$ and some small enough $\delta>0$. Then any $\tilde\tau\neq\tau$ for which $z_k\in \partial S_{Q,\tau_k}$ necessarily satisfies $\tilde\tau_1+\ldots+\tilde\tau_d<1$, which implies that
    \[
    z\not\in \bigcup_{\substack{\tau_1, \, \ldots \, ,\tau_d \ge 0 \\ \tau_1+\cdots+\tau_d= 1}}
\;
\prod_{k=1}^d \partial S_{Q,\tau_k} \,.
    \]
    A similar argument works for the exterior. 
\end{proof}

\begin{example}
Let us consider $\mathscr Q(z)=a_1|z_1|^2+\ldots+a_d |z_d|^2$ for some constants $a_1, \ldots, a_d>0$. Then $S_{Q_k, \tau_k}= \{z_k\in\mathbb C : a_k |z_k|^2\leq \tau_k\}$ and
\[
\check Q_{k,\tau_k}(z_k) = \tau_k+\tau_k \log|z_k|^2-\tau_k\log \frac{\tau_k}{a_k}.
\]
A standard Lagrange multiplier approach then yields for large enough $|z|$
\[
\check{\mathscr Q}(z) = 1+\log(a_1|z_1|^2+\ldots+a_d|z_d|^2).
\]
We infer that
\[
S_{\mathscr Q} = \{z\in\mathbb C^d: a_1|z_1|^2+\ldots+a_d|z_d|^2\leq 1\}
\]
with $\check{\mathscr Q}(z)$ given by the preceding formula when $z\in \mathbb C\setminus S_{\mathscr Q}$ and by $\mathscr Q$ when $z\in S_{\mathscr Q}$.
\end{example}

\begin{example}
Consider the pluripotential elliptic Ginibre ensemble, for convenience scaled as $\mathscr Q(z)=\frac1{1-\tau^2}(|z|^2-\tau \mathrm{Re}\, \sum_{k=1}^d z_k^2)$. Here, as proved in \cite{AkemannDuitsMolag2023} the droplet is given by the hyperellipsoid
\[
S_{\mathscr Q} = \left\{z\in\mathbb C^d :\frac{|\mathrm{Re}\, z|^2}{(1+\tau)^2}+\frac{|\mathrm{Im}\, z|^2}{(1-\tau)^2}\leq 1\right\}.
\]
One obtains from \cite[Theorem 3.7]{Berman2009}
 combined with \cite[Proposition II.3 and Lemma V.1]{AkemannDuitsMolag2023} that
\[
\check{\mathscr Q}(z) = 
\log |\Psi(z)|^2 +2+ 2\tau \mathrm{Re}\, \frac{1}{\Psi(z)^2},
\]
on $\mathbb C^d\setminus S_{\mathscr Q}$, where
\[
\Psi(z) = \frac{|\mathrm{Re}\, z|+i|\mathrm{Im}\, z| + \sqrt{(|\mathrm{Re}\, z|+i|\mathrm{Im}\, z|)^2-4\tau}}{\sqrt 2}.
\]
(See also \cite[Proposition VI.1]{AkemannDuitsMolag2026}.) 
\end{example}

\subsection{Local edge scaling limits of the kernel}

Finally, let us investigate the weighted polynomial Bergman kernel at the edge. In the case of a factorized weight, it takes the form
\begin{align*}
\mathscr K_n(z,w)=
    e^{-\frac12n \mathscr Q(z)}e^{-\frac12n \mathscr Q(w)}\sum_{|j|<n} \mathscr P_j(z) \overline{\mathscr P_j(w)}
\end{align*}
where $|j|<n$ denotes summation over indices $j=(j_1,\ldots,j_d)\in\mathbb \{0,\ldots,n-1\}^d$ such that $|j|=j_1+\ldots+j_d<n$, and $\mathscr P_{j}(z)$ are the multivariate polynomials
\[
\mathscr P_j(z) = \prod_{k=1}^d P_{k,j_k}(z_k),
\]
where $P_{k,\ell}$ are the planar orthogonal polynomials of degree $\ell$ and positive leading coefficient satisfying the orthogonality conditions
\[
\int_{\mathbb C} P_{k,\ell}(z) \overline{P_{k,\ell'}(z)} e^{-n Q_k(z)} \, dA(z)=\delta_{\ell,\ell'}, \qquad \ell, \ell'=0,1,\ldots
\]
Note that the polynomials $\mathscr P_j(z)$ are orthonormal to each other with respect to the weight $e^{- n \mathscr Q(z)}$ on $\mathbb C^d$.\\

In the seminal paper \cite{HedenmalmWennman2021}, Hedenmalm and Wennman proved an asymptotic formula for the orthogonal polynomials $P_j:\mathbb C\to \mathbb C$ (we supress the $n$-dependence) of degree $j$ and with positive leading coefficient, satisfying the relations
\begin{align*}
\int_{\mathbb C} P_j(z) \overline{P_k(z)} e^{-n Q(z)} dA(z) = \delta_{jk},
\qquad j,k=0,1,\ldots
\end{align*}
when $Q$ is $1$-admissible. For any integer $\kappa>0$, there is an expansion formula
\[
    P_{j}(z) = n^{1/4}[\phi_{\tau}'(z)]^{1/2}[\phi_{\tau}(z)]^{j}
    e^{\frac12n\mathcal{Q}_{\tau}(z)}
    \left(\sum_{\ell=0}^{\kappa} n^{-\ell}\mathcal{B}_{\tau,\ell}(z)
    + \mathrm{O}(n^{-\kappa-1})\right),
\]
where the error term is uniform over all $z \in \mathbb{C}$ with
\[
    \mathrm{dist}_{\mathbb{C}}(z,\mathcal{S}_{Q,\tau}^{c})
    \leqslant A(n^{-1}\log n)^{1/2}
\]
as $j = \tau n \to +\infty$ along the integers such that $\tau\in(1-\epsilon,1+\epsilon)$, for some small enough $\epsilon>0$. Here $A>0$ is allowed to be any fixed constant. $\phi_\tau$ is the orthostatic (meaning $\phi_\tau(\infty)=\infty$ and $\phi_\tau'(\infty)>1$) conformal map from the exterior of the $\tau$-droplet $S_{Q,\tau}$ to the exterior of the unit disc.  $\mathcal Q_\tau$ is the bounded holomorphic function on (a neighborhood of) $\mathbb C\setminus S_ {Q,1}$ whose real part agrees with $Q$ on $\partial S_{Q,1}$ with imaginary part vanishing at infinity. The $\mathcal B_{\tau,\ell}$ are bounded holomorphic functions on some fixed neighborhood of $\mathbb C\setminus S_{Q,1}$. We shall only need the first one, which has modulus squared
\[
|\mathcal B_{\tau,\ell}(z)|^2 = \frac1{\sqrt \pi} \sqrt{\Delta Q(z)}.
\]
We thus have
\begin{align*}
\sqrt\pi e^{-n Q(z)}\left|P_{j}(z)\right|^2 = \sqrt{n \Delta Q(z)} |\phi_{\tau}'(z)||\phi_{\tau}(z)|^{2j} e^{n(\mathcal Q_\tau(z)-Q(z))} \left(1+\mathcal O(1/n)\right)
\end{align*}
uniformly for $\mathrm{dist}_{\mathbb{C}}(z,\mathcal{S}_{Q,\tau}^{c})
    \leqslant A(n^{-1}\log n)^{1/2}$, as $j = \tau n \to +\infty$ along the integers such that $\tau\in(1-\epsilon,1+\epsilon)$. Now let $\tau_0>0$. 
It is a straightforward consequence of the minimization problem \eqref{eq:MinProb} that the $1$-droplet of the planar potential $\tau_0^{-1} Q$ is given by $S_{Q,\tau_0}$. We thus have a similar expansion for $\tau_0$-admissible potentials $Q$, where we get 
\begin{align*}
\sqrt\pi e^{-n Q(z)}\left|P_{j}(z)\right|^2 
&= \sqrt\pi e^{-(\tau_0 n) \tau_0^{-1} Q(z)}\left|P_{j}(z)\right|^2\\
&= \sqrt{n \Delta Q(z)} |\phi_{\tau}'(z)||\phi_{\tau}(z)|^{2j} e^{n(\mathcal Q_\tau(z)-Q(z))} \left(1+\mathcal O(1/n)\right)
\end{align*}
uniformly for $\mathrm{dist}_{\mathbb{C}}(z,\mathcal{S}_{Q,\tau}^{c})
    \leqslant A(n^{-1}\log n)^{1/2}$, as $j = \tau n \to +\infty$ along the integers such that $\tau\in(\tau_0-\epsilon,\tau_0+\epsilon)$, perhaps with different constants $A>0$ and $\epsilon>0$. Now let $z_0\in \partial S_{Q,\tau_0}$. It was proved in \cite{HedenmalmWennman2021} that there exists some constant $c_0>0$ (independent ot $z_0$) such that for all integers $j<n-a \sqrt n\log n$
\begin{align} \label{eq:estPolySmallj}
\left|P_j\left(z_0+\frac{\vec n_{\tau_0}(z_0) \xi}{\sqrt{n \Delta Q(z_0)}}\right)\right|^2 \exp{\left(-n Q\left(z_0+\frac{\vec n_{\tau_0}(z_0) \xi}{\sqrt{n \Delta Q(z_0)}}\right)\right)}
= \mathcal O(n e^{- c_0\log^2 n}),
\end{align}
uniformly for $\xi\in\mathbb C$ with $|\xi|=\mathcal O(\sqrt{\log n})$, where the implied constant does not depend on our choice of $z_0$. This was strictly speaking proved for $a=1$, but it is easily seen to hold for any fixed $a>0$. With a similar argument, this estimate also holds for $j>n+a\sqrt n\log n$. 
 Furthermore, for all integers $\tau_0 n -a \sqrt n \log n\leq j<n$ the results in \cite[Section 5]{HedenmalmWennman2021} imply that
\begin{multline} \label{eq:nearBoundaryPoly}
\sqrt\pi\left|P_j\left(z_0+\frac{\vec n_{\tau_0}(z_0) \xi}{\sqrt{n \Delta Q(z_0)}}\right)\right|^2 \exp{\left(-n Q\left(z_0+\frac{\vec n_{\tau_0}(z_0) \xi}{\sqrt{n \Delta Q(z_0)}}\right)\right)}\\
= \sqrt{n \Delta Q_k(z_{0,k})} |\phi'_{k,\tau_k}(z_{0,k})|\exp\!\left(
  -\frac{1}{2}
  \left(
    2\,\mathrm{Re}\,\xi_k
    + (\tau_k n-j_k)\,\frac{|\phi_{k,\tau_k}'(z_{0,k})|}{\sqrt{n\,\Delta Q_k(z_{0,k})}}
  \right)^2
\right)\\
(1+ \mathcal O(\log^3 n/\sqrt n))
\end{multline}
uniformly for $|\xi|=\mathcal O(\sqrt{\log n})$. Again, one can extend such behavior to indices with $j>n+a\sqrt n\log n$. 

Let us now consider the general case $d\geq 1$. 
We shall first consider the case where $\xi=\eta\in\mathbb C^d$. By Proposition \ref{thm:dropletBoundary}, any point $z_0\in\partial S_{\mathscr Q}$ is of the form $z_0=(z_{0,1}, \ldots, z_{0,d})$ where $z_{0,k}\in \partial S_{Q,\tau_k}$ for all $k=1,\ldots,d$. Let us first consider the case that $\tau_1, \ldots, \tau_d>0$. In what follows we denote for each $k=1,\ldots,d$ by
\[
\vec n_{\tau_k}(z_{0,k}), \qquad z_{0,k} \in \partial S_{Q_k,\tau_k}
\]
the outward unit normal vector at $z_{0,k}$ on $\partial S_{Q_k,\tau_k}$. 

\begin{lemma} \label{lem:estSumallTauBigger0}
Suppose that $\mathscr Q:\mathbb C^d\to\mathbb R$ decomposes as a sum of $[0,1]$-admissible planar potentials.
Assume that $z_0\in \partial S_{Q_1,\tau_1}\times \cdots \times \partial S_{Q_d,\tau_d}$ where $\tau_1, \ldots, \tau_d>0$ and $\tau_1+\ldots+\tau_d=1$.
Let $\mathscr U(z_0)$ be the unitary matrix $\operatorname{diag}(\vec n_{\tau_1}(z_{0,1}),\ldots, \vec n_{\tau_d}(z_{0,d}))$. Then we have
    \begin{multline*}
        \frac1{\det n \partial\bar\partial\mathscr Q(z_0)} \mathscr K_n\left(z_0+\frac{\mathscr U(z_0)\xi}{\sqrt{n \partial\bar\partial\mathscr Q(z_0)}}\right)\\
        =\left(1+\mathcal O(1/n)\right) \frac1{\pi^{d/2}}
        \sum_{\substack{i_1+\ldots+i_d>0\\ -\sqrt n\log n<i_1,\ldots,i_d\leq \sqrt n \log n}} \, \prod_{k=1}^d
        \frac{|\phi'_{k,\tau_k}(z_{0,k})|}{\sqrt{n \Delta Q_k(z_{0,k})}}\\
        \exp\!\left(
  -\frac{1}{2}
  \left(
    2\,\mathrm{Re}\,\xi_k
    + i_k\,\frac{|\phi_{k,\tau_k}'(z_{0,k})|}{\sqrt{n\,\Delta Q_k(z_{0,k})}}
  \right)^2
\right)
    \end{multline*}
    uniformly for all $\xi\in\mathbb C^d$ with $|\xi|=\mathcal O(\sqrt{\log n})$.
\end{lemma}

\begin{proof}
By the discussion surrounding \eqref{eq:estPolySmallj}, we may exclude terms such that $\tau_k n-\sqrt n \log n<j_n<\tau_k n+\sqrt n \log n$ from the sum defining the correlation kernel, assuming $n$ is big enough. 
Since the number of such terms is clearly less than $n^d$, and each individual weighted polynomial is bounded by the kernel, and hence $\mathcal O(n)$, we find that
\begin{multline*}
\mathscr K_n\left(z_0+\frac{\mathscr U(z_0)\xi}{\sqrt{n \partial\bar\partial \mathscr Q(z_0)}}\right) = \mathcal O(e^{- c \log^2n}) +
\sum_{\substack{j_1+\ldots+j_d<n\\ |j_k-\tau_k n| \leq\sqrt n\log n}}
\prod_{k=1}^d
\\
\left|P_{j_k}\left(z_{0,k}+\frac{\vec n_{\tau_k}(z_{0,k}) \xi_k}{\sqrt{n \Delta Q_k(z_{0,k})}}\right)\right|^2 \exp{\left(-n Q_k\left(z_{0,k}+\frac{\vec n_{\tau_k}(z_{0,k}) \xi_k}{\sqrt{n \Delta Q_k(z_{0,k})}}\right)\right)}
\end{multline*}
for some suitably chosen $c>0$, uniformly for $|\xi|=\mathcal O(\sqrt{\log n})$. Inserting the behavior \eqref{eq:nearBoundaryPoly} we see that
\begin{multline*}
\mathscr K_n\left(z_0+\frac{\mathscr U(z_0)\xi}{\sqrt{n \partial\bar\partial \mathscr Q(z_0)}}\right) = \mathcal O(e^{- c \log^2n}) +\\
\left(1+\mathcal O(1/n)\right) \frac1{\pi^{d/2}}
\sum_{\substack{j_1+\ldots+j_d<n\\ |j_k - \tau_k n|\leq \sqrt n\log n}}
\prod_{k=1}^d
        \frac{|\phi'_{k,\tau_k}(z_{0,k})|}{\sqrt{n \Delta Q_k(z_{0,k})}}\\
        \exp\!\left(
  -\frac{1}{2}
  \left(
    2\,\mathrm{Re}\,\xi_k
    + (\tau_k n - j_k)\,\frac{|\phi_{k,\tau_k}'(z_{0,k})|}{\sqrt{n\,\Delta Q_k(z_{0,k})}}
  \right)^2
\right)
    \end{multline*}
    uniformly for all $\xi\in\mathbb C^d$ with $|\xi|=\mathcal O(\sqrt{\log n})$.
Relabelling $i_k=\lceil\tau_k n\rceil-j_k$ we obtain the result. (Note that any missed or added index due to the rounding gives an error of order $e^{-c\log^2n}$ for some $c>0$.)
\end{proof}

This multidimensional sum is seen to be a Riemann sum. Effectively, we replace
\[
\frac{|\phi_{k,\tau_k}'(z_{0,k})|}{\sqrt{n\,\Delta Q_k(z_{0,k})}}
\to x_k, \qquad k=1,\ldots,d,
\]
and we obtain a multidimensional integral over the polytope that is bounded by the boundary of the hypercube $[-1,1]^d$ and the plane $x_1+\ldots+x_d=0$. Explicitly, the error terms can be expressed as integrals over the faces of the polytope \cite{BerlineVergne2007, GuilleminSternberg2007}, and in our case the important thing is that
\begin{multline} \label{eq:EulerMac}
\sum_{\substack{i_1+\ldots+i_d>0\\-\sqrt n\log n\leq i_1,\ldots,i_d\leq \sqrt n\log n}}
\prod_{k=1}^d
        \frac{|\phi'_{k,\tau_k}(z_{0,k})|}{\sqrt{n \Delta Q_k(z_{0,k})}}\\
        \exp\!\left(
  -\frac{1}{2}
  \left(
    2\,\mathrm{Re}\,\xi_k
    + i_k\,\frac{|\phi_{k,\tau_k}'(z_{0,k})|}{\sqrt{n\,\Delta Q_k(z_{0,k})}}
  \right)^2
\right)\\
= \int_{x_1+\ldots+x_d\geq 0} \exp\left(-\frac12\left(2 \mathrm{Re} \xi_k+x_k\right)^2\right) dx_1 \cdots dx_d\\
+ \mathcal O\left(\frac{e^{- 2 |\mathrm{Re}\,\xi|^2}}{\sqrt n}\right)
\end{multline}
uniformly for $|\xi|=\mathcal O(n^{\frac12-\epsilon})$ for any fixed $\epsilon>0$ as $n\to\infty$, and certainly for $|\xi|=\mathcal O(\sqrt{\log n})$. Here the implied constant can be taken independently of $z_0$, which follows from the continuity of $\phi_{k,\tau_k}'(z_0)$ and $\Delta Q_k(z_{0,k})$ and the compactness of $\partial S_{\mathscr Q}$. 
We analyze this integral explicitly in Lemma \ref{lem:GaussianErfc} in Appendix \ref{sec:A}. Combining \eqref{eq:EulerMac} with Lemma \ref{lem:GaussianErfc}, we get the following corollary. 

\begin{corollary} \label{cor:SumToInt}
Suppose that $\mathscr Q:\mathbb C^d\to\mathbb R$ decomposes as a sum of $[0,1]$-admissible planar potentials.
Assume that $z_0\in \partial S_{Q_1,\tau_1}\times \cdots \times \partial S_{Q_d,\tau_d}$ where $\tau_1, \ldots, \tau_d>0$ and $\tau_1+\ldots+\tau_d=1$.
Let $\mathscr U(z_0)$ be the unitary matrix 
$\operatorname{diag}(\vec n_{\tau_1}(z_{0,1}),\ldots, \vec n_{\tau_d}(z_{0,d}))$. Then we have as $n\to\infty$ that
    \begin{multline*}
        \frac1{\det n \partial\bar\partial\mathscr Q(z_0)} \mathscr K_n\left(z_0+\frac{\mathscr U(z_0)\xi}{\sqrt{n \partial\bar\partial\mathscr Q(z_0)}}\right)\\
        =
        \frac12 \mathrm{erfc}\left(\sqrt 2 \mathrm{Re} \, \sum_{k=1}^d \frac{\xi_k}{\sqrt {d}}\right)
        + \mathcal O\left(\frac{e^{-2 |\mathrm{Re}\,\xi|^2}}{\sqrt n}\right)
    \end{multline*}
    uniformly for all $\xi\in\mathbb C^d$ with $|\xi|=\mathcal O(\sqrt{\log n})$. 
\end{corollary}

\begin{proof}
We let $A$ be the $d\times d$ identity matrix and $b=2 \, \mathrm{Re} \, \xi\in \mathbb R^d$ in \text{Lemma \ref{lem:GaussianErfc}}.
\end{proof}

The remaining cases exhibit a certain \textit{bulk degeneracy}: one or more coordinates of $z_0\in\partial S_{\mathscr Q}$ may be in $S_{Q_k,0}=\{p_{Q_k}\}$. Then $p_{Q_k}$ is an interior point of $S_{Q_k,\tau}$ for all $\tau>0$ but becomes a boundary point for $\tau=0$. 
Still in the diagonal setting $\xi=\eta$, we now turn to the case where several of the $\tau_k$ might be $0$. This situation has to be treated with care. 
A key role is played by Theorem \ref{thm:PartialBerg} which applies to $[0,1]$-admissible potentials $Q:\mathbb C\to\mathbb R$. It implies that for any nonnegative integer $m_n$ of order $\sqrt n \log n$ 
\begin{align*}
\frac{e^{-n Q(p_Q+\xi/\sqrt{n \Delta Q(p_Q)})}}{\sqrt{n \Delta Q(p_Q)}} \sum_{j=0}^{m_n} \left|P_j\left(p_Q+\frac{\xi}{\sqrt{n \Delta Q(p_Q)}}\right)\right|^2
= 1+\mathcal O\left(\frac{\log n}{\sqrt n}\right)
\end{align*}
uniformly for $\xi\in\mathbb C$ with $|\xi|=\mathcal O(\sqrt{\log n})$ as $n\to\infty$, where the implied constant depends only on $Q$.  
By translation, we assume henceforth without loss of generality that
\[
p_{Q_k} = 0, \qquad k=1,\ldots,d.
\]
Assume that $\tau_1, \ldots, \tau_\ell>0$ and $\tau_{\ell+1}, \ldots, \tau_d=0$. 

\begin{lemma} \label{lem:bulkDegen}
Suppose that $\mathscr Q:\mathbb C^d\to\mathbb R$ decomposes as a sum of $[0,1]$-admissible planar potentials.
For each $z_0\in\partial S_{\mathscr Q}$, there exists a unitary matrix $\mathscr U(z_0)$ such that as $n\to\infty$
    \begin{multline*}
        \frac1{\det n \partial\bar\partial\mathscr Q(z_0)} \mathscr K_n\left(z_0+\frac{\mathscr U(z_0)\xi}{\sqrt{n \partial\bar\partial\mathscr Q(z_0)}}\right)\\
        =
        \frac12 
        \mathrm{erfc}\left(\sqrt 2 \mathrm{Re} \, \sum_{k=1}^d \frac{\xi_k}{\sqrt {d}}\right) 
        \left(1+\mathcal O\left(\frac{\log n}{\sqrt n}\right)\right)
    \end{multline*}
    uniformly for all $\xi\in\mathbb C^d$ with $|\xi|=\mathcal O(\sqrt{\log n})$, where the implied constant is independent of $z_0$.
\end{lemma}

\begin{proof}
We may assume that $\ell>0$. 
Assume that $z_0\in \partial S_{Q,\tau_1}\times \cdots \times \partial S_{Q,\tau_d}$ where $\tau_1+\ldots+\tau_d=1$, and (without loss of generality) $\tau_1, \ldots, \tau_m>0$ and $\tau_{m+1}=\cdots=\tau_d=0$. We shall denote $\vec n_0(z_{0,k})=1$ in what follows, i.e., for the indices $k=\ell+1,\ldots,d$. 
We have an obvious upper bound
\begin{multline*}
\sum_{j_1+\ldots+j_d<n} \prod_{k=1}^d
\left|P_{j_k}\left(z_{0,k}+\frac{\vec n_{\tau_k}(z_{0,k}) \xi_k}{\sqrt{n \Delta Q_k(z_{0,k})}}\right)\right|^2\\ 
\exp{\left(-n Q_k\left(z_{0,k}+\frac{\vec n_{\tau_k}(z_{0,k}) \xi_k}{\sqrt{n \Delta Q_k(z_{0,k})}}\right)\right)}\\
\leq 
\sum_{j_1+\ldots+j_\ell<n} \prod_{k=1}^\ell
\left|P_{j_k}\left(z_{0,k}+\frac{\vec n_{\tau_k}(z_{0,k}) \xi_k}{\sqrt{n \Delta Q_k(z_{0,k})}}\right)\right|^2\\ 
\exp{\left(-n Q_k\left(z_{0,k}+\frac{\vec n_{\tau_k}(z_{0,k}) \xi_k}{\sqrt{n \Delta Q_k(z_{0,k})}}\right)\right)}
\end{multline*}
(i.e., we simply bound the sums over $j_{\ell+1}, \ldots, j_d$ by the full sums from $0$ to $\infty$, which equal $1$.)
On the other hand, for any fixed $0<\epsilon<1/d$, we have the lower bound
\begin{multline*}
\sum_{j_1+\ldots+j_d<n}
\left|P_{j_k}\left(z_{0,k}+\frac{\vec n_{\tau_k}(z_{0,k}) \xi_k}{\sqrt{n \Delta Q_k(z_{0,k})}}\right)\right|^2 \exp{\left(-n Q_k\left(z_{0,k}+\frac{\vec n_{\tau_k}(z_{0,k}) \xi_k}{\sqrt{n \Delta Q_k(z_{0,k})}}\right)\right)}\\
\geq 
\sum_{\substack{j_{\ell+1},\ldots,j_d\\<\epsilon \sqrt n\log n}}
\sum_{\substack{j_1+\ldots+j_\ell\\<n-(j_{\ell+1}+\ldots+j_d)}} \prod_{k=1}^n
\left|P_{j_k}\left(z_{0,k}+\frac{\vec n_{\tau_k}(z_{0,k}) \xi_k}{\sqrt{n \Delta Q_k(z_{0,k})}}\right)\right|^2\\
\exp{\left(-n Q_k\left(z_{0,k}+\frac{\vec n_{\tau_k}(z_{0,k}) \xi_k}{\sqrt{n \Delta Q_k(z_{0,k})}}\right)\right)}
\end{multline*}
For any fixed index $(j_1,\ldots,j_\ell)$ with $0\leq j_1,\ldots,j_\ell<\epsilon \sqrt n\log n$ we have for the inner sum
\begin{multline*}
\sum_{\substack{j_1+\ldots+j_\ell\\<n-(j_{\ell+1}+\ldots+j_d)}} \prod_{k=1}^\ell
\left|P_{j_k}\left(z_{0,k}+\frac{\vec n_{\tau_k}(z_{0,k}) \xi_k}{\sqrt{n \Delta Q_k(z_{0,k})}}\right)\right|^2\\ 
\exp{\left(-n Q_k\left(z_{0,k}+\frac{\vec n_{\tau_k}(z_{0,k}) \xi_k}{\sqrt{n \Delta Q_k(z_{0,k})}}\right)\right)}\\
\geq
\sum_{\substack{j_1+\ldots+j_\ell\\<(1-\epsilon d)n}} \prod_{k=1}^\ell
\left|P_{j_k}\left(z_{0,k}+\frac{\vec n_{\tau_k}(z_{0,k}) \xi_k}{\sqrt{n \Delta Q_k(z_{0,k})}}\right)\right|^2\\ 
\exp{\left(-n Q_k\left(z_{0,k}+\frac{\vec n_{\tau_k}(z_{0,k}) \xi_k}{\sqrt{n \Delta Q_k(z_{0,k})}}\right)\right)}.
\end{multline*}
We already know how to estimate these sums. Namely, by Lemma \ref{lem:estSumallTauBigger0} and Corollary \ref{cor:SumToInt}, applied for $\ell$ instead of $d$, and $(1-\epsilon d)n$ instead of $n$, we have
\begin{multline*}
\sum_{\substack{j_1+\ldots+j_\ell\\<(1-\epsilon)n}} \prod_{k=1}^\ell
\left|P_{j_k}\left(z_{0,k}+\frac{\vec n_{\tau_k}(z_{0,k}) \xi_k}{\sqrt{n \Delta Q_k(z_{0,k})}}\right)\right|^2\\ 
\exp{\left(-n Q_k\left(z_{0,k}+\frac{\vec n_{\tau_k}(z_{0,k}) \xi_k}{\sqrt{n \Delta Q_k(z_{0,k})}}\right)\right)}\\
= \frac12 \mathrm{erfc}\left(\sqrt 2 \, \mathrm{Re} \sum_{k=1}^\ell \frac{\xi_k}{\sqrt{\ell}}\right)+\mathcal O\left(\frac{e^{-2 |\mathrm{Re}\,\xi|^2}}{\sqrt n}\right)
\end{multline*}
uniformly for $|\xi|=\mathcal O(\sqrt{\log n})$ as $n\to\infty$, where the constant implied by the $\mathcal O$ term is independent of $(j_1,\ldots,j_\ell)$. Now applying Theorem \ref{thm:PartialBerg} to the remaining sums over $j_1,\ldots,j_\ell$ we finally get the lower bound
\begin{multline*}
\sum_{j_1+\ldots+j_d<n}
\left|P_{j_k}\left(z_{0,k}+\frac{\vec n_{\tau_k}(z_{0,k}) \xi_k}{\sqrt{n \Delta Q_k(z_{0,k})}}\right)\right|^2 \exp{\left(-n Q_k\left(z_{0,k}+\frac{\vec n_{\tau_k}(z_{0,k}) \xi_k}{\sqrt{n \Delta Q_k(z_{0,k})}}\right)\right)}\\
\geq \left(1-C\frac{\log n}{\sqrt n}\right)
\left(\frac12 \mathrm{erfc}\left(\sqrt 2 \, \mathrm{Re} \sum_{k=1}^\ell \frac{\xi_k}{\sqrt{\ell}}\right)-C\frac{e^{-2|\mathrm{Re}\,\xi|^2}}{\sqrt n}\right)\\
\geq \left(1-C'\frac{\log n}{\sqrt n}\right)
\frac12 \mathrm{erfc}\left(\sqrt 2 \, \mathrm{Re} \sum_{k=1}^\ell \frac{\xi_k}{\sqrt{\ell}}\right)
\end{multline*}
uniformly for $|\xi|=\mathcal O(\sqrt{\log n})$ as $n\to\infty$, for some constants $C, C'>0$.
Now let $\mathscr D(z_0)=\operatorname{diag}(\vec n_{\tau_1}(z_{0,1}),\ldots, \vec n_{\tau_{\ell}}(z_{0,\ell}),1,\ldots,1))$. What we have proved at this point is that
\begin{multline*}
        \frac1{\det n \partial\bar\partial\mathscr Q(z_0)} \mathscr K_n\left(z_0+\frac{\mathscr D(z_0)\xi}{\sqrt{n \partial\bar\partial\mathscr Q(z_0)}}\right)\\
        =
        \frac12 
        \mathrm{erfc}\left(\sqrt 2 \mathrm{Re} \, \sum_{k=1}^\ell \frac{\xi_k}{\sqrt {d}}\right) 
        \left(1+\mathcal O\left(\frac{\log n}{\sqrt n}\right)\right)
    \end{multline*}
    uniformly for all $\xi\in\mathbb C^d$ with $|\xi|=\mathcal O(\sqrt{\log n})$, where the implied constant is independent of $z_0$.
For the final step, let $\mathscr R(z_0)$ be the unitary (rotation) matrix that sends the unit vector $\frac1{\sqrt{\ell}}(1,\ldots,1,0,\ldots,0)$ to $\frac1{\sqrt d} (1, \ldots, 1)$. The result follows when we take the unitary matrix $\mathscr U(z_0)=\mathscr D(z_0) \mathscr R(z_0)$.
\end{proof}

In principle, the approach considered above, using the Euler-Maclaurin formula can be used for the off-diagonal case as well. For example in \cite{MarzoMolagOrtegaCerda2026} it was shown that, uniformly for $\xi,\eta\in\mathbb C$ and $\alpha\in\mathbb R$, we have
    \begin{multline*}
        \sum_{j=1}^{m_n} \frac{|\phi_1'(z_0)|}{2\sqrt{n\Delta Q(z_0)}}
     e^{-\left(\mathrm{Re} \, \xi+\frac{j|\phi_1'(z_0)|}{2\sqrt{n\Delta Q(z_0)}}\right)^2
     -\left(\mathrm{Re} \, \eta+\frac{j|\phi_1'(z_0)|}{2\sqrt{n\Delta Q(z_0)}}\right)^2+i\sqrt n \alpha \frac{j|\phi_1'(z_0)|}{2\sqrt{n\Delta Q(z_0)}}}\\
     = \int_0^{\infty} e^{-(\mathrm{Re} \, \xi+t)^2-(\mathrm{Re} \, \eta+t)^2+i\sqrt n \alpha t} dt
-\frac{|\phi_1'(z_0)|}{4\sqrt{n\Delta Q(z_0)}} e^{-(\mathrm{Re} \, \xi)^2-(\mathrm{Re} \, \eta)^2}\\
     + \mathcal O\Big(\frac{(|\mathrm{Re} \, \xi|+|\mathrm{Re}\,\eta|+\sqrt n|\alpha|+\log n)^2}{n}\log n\Big).
    \end{multline*}
    where $\alpha$ can be expressed with the help of the conformal map $\phi_1$ and is in general nonzero when $\xi\neq\eta$.\footnote{We correct here for a typo in Lemma 4.2 in \cite{MarzoMolagOrtegaCerda2026}.} Note a similar derivation in \cite{Charlier2025}. However, this approach becomes significantly more technical in the setting $d>1$.

It is cleaner and somewhat more satisfying to extend our results by a polarization argument.

\begin{proof}[Proof of Theorem \ref{thm:generalMain}]
For $k=1,\ldots,d$, we define the functions
    \[
    K_{n,k}(z,w) = e^{-\frac12 Q_k(z)} e^{-\frac12 Q_k(w)} \sum_{j=0}^{n-1} P_{k,j}(z) \overline{P_{k,j}(w)},
    \qquad z,w\in\mathbb C,
    \]
    where $P_{k,j}$ are the planar orthogonal polynomials corresponding to $Q_k$, and
    \[
    K^{\#}_{n,k}(z,w) = \partial_1 \bar\partial_2 Q_k(z,\overline z) e^{n Q_k(z,\overline w)},
    \qquad z,w\in\mathbb C,
    \]
    where $Q_k(\cdot, \cdot)$ denotes the polarization.
    It was proved in \cite{AmeurKangMakarov2019} that the functions
    \begin{align*}
        \psi_{n,k}(z,w)=\frac{K_{n,k}(z,w)}{K^{\#}_{n,k}(z,w)}
    \end{align*}
    are Hermitian-analytic. Then this is also true for the function
    \begin{align*}
        \Psi_n(z,w) = \frac{\mathscr K_n(z,w)}{K_{n,1}(z_1,w_1)\cdots K_{n,d}(z_d,w_d)}
        \prod_{k=1}^d \psi_{n,k}(z_d,w_d).
    \end{align*}
    (Note that the weight factors cancel in the quotient involving $\mathscr K_n$.) Using Cauchy-Schwarz, it is clear that the expression is locally bounded when we rescale variables. Hence by Vitali's theorem we know that
    \[
    \lim_{n\to\infty} \frac1{\det n \partial\bar\partial\mathscr Q(z_0)} \Psi_n\left(z_0+\frac{\mathscr U(z_0)\xi}{\sqrt{n \partial\bar\partial\mathscr Q(z_0)}} ,z_0+\frac{\mathscr U(z_0) \eta}{\sqrt{n \partial\bar\partial\mathscr Q(z_0)}}\right)=\Psi(\xi,\eta),
    \]
    where $\Psi$ is some Hermitian-analytic function in a neighborhood of the diagonal $\xi=\eta$. When $\xi=\eta$ we already know that $\Psi(\xi,\xi)=\mathrm{erfc}(\sqrt 2 \mathrm{Re} \xi)$. By analytic continuation we must then have $\Psi(\xi,\eta)=\mathrm{erfc} \frac{\xi+\overline\eta}{\sqrt 2}$. Since the convergence on the diagonal holds for $|\xi|=\mathcal O(\sqrt{\log n})$, the off-diagonal convergence holds for $|\xi|,|\eta|=\mathcal O(\sqrt{\log n})$. (One may simply rescale the variables of the above functions by a factor $\sqrt{\log n}$.)
    The missing factor, the pluricomplex Ginibre kernel, follows simply by a Taylor expansion of $n \mathscr Q$. We have proved that \eqref{eq:con2} holds uniformly for $z_0\in\partial S_{\mathscr Q}$ and $|\xi|,|\eta|=\mathcal O(\sqrt{\log n})$ as $n\to\infty$. 
    
    The remaining part, equation \eqref{eq:con1}, now follows from Proposition \ref{lem:unitVector}. 
\end{proof}

\section{Rotational symmetric weights} \label{sec:3}

In this section we assume that $\mathscr Q:\mathbb C^d\to\mathbb R$ is of the form
\[
\mathscr Q(z)=V(|z|),
\]
where $V$ is supposed to satisfy certain growth and regularity conditions. In particular, since we want $\mathscr Q$ to satisfy \eqref{eq:growthCond}, we necessarily have
\begin{align} \label{eq:growthCondV}
    \liminf_{r\to\infty} \frac{V(r)}{\log r} > 2.
\end{align}
We shall also assume that $V$ is $C^2$ on $(0,\infty)$, that $r\mapsto r V'(r)$ is strictly increasing and that
\begin{align} \label{eq:rV'}
\lim_{r\to 0} r V'(r) = 0.
\end{align}
Note that there is no issue in the integrals defining the orthogonality relations in \eqref{eq:defOP}, since \eqref{eq:rV'} implies that $e^{-n V(r)}=\mathcal O(r^\alpha)$ for any fixed $\alpha\in(-1,0)$ as $r\to 0$.
These conditions are equivalent to the conditions of Theorem \ref{thm:generalMain2}, as the reader may verify.
Additionally, it will be convenient to introduce the planar potential $Q:\mathbb C\to\mathbb R$ defined by $Q(z)=V(|z|)$.
As mentioned in \cite{AkemannByunEbke2023} (see also \cite{SaffTotik1997}) the conditions imply that the droplet $S_Q$ is simply connected, i.e., a disk, and $S_Q=S_Q^\star$. It is explicitly given by
\[
S_{Q} = \{z\in\mathbb C : |z| V'(|z|)\leq 2\}.
\]
Without loss of generality, we impose the normalizing condition
\[
V'(1)=2,
\]
which implies that $S_Q$ is the unit disk. We shall prove in Proposition \ref{eq:rotationalCheckQ} below that, as expected, one finds
\[
\check{\mathscr Q}(z) = \check Q(|z|), \qquad z\in\mathbb C^d,
\]
and the droplet is given by
\[
S_{\mathscr Q} = \{z\in\mathbb C^d : |z|\leq 1\},
\]
the closed $2d$-dimensional unit ball in $\mathbb C^d$. 
Henceforth, we let $z_0\in\partial S_{\mathscr Q}$, i.e., $|z_0|=1$.

\subsection{Local edge scaling limits of the kernel}

First, we derive some identities for $\partial \bar\partial \mathscr Q$ and related expressions that are needed for the local scaling limits. 
With straightforward calculations, one may show that
\begin{align} \label{eq:MongeAmpereRot}
\partial\bar\partial \mathscr Q(z) &= \frac{V'(|z|)}{2|z|}\mathbb I+\frac{|z| V''(|z|)-V'(|z|)}{4|z|^3} z z^\dagger.
\end{align}
We can express $(\partial\bar\partial \mathscr Q(z))^{-1/2}$ explicitly, using the rank $1$ structure. For $|z_0|=1$ we get 
\begin{align} \nonumber
(\partial\bar\partial \mathscr Q(z_0))^{-1/2} &= \left(\mathbb I+\frac{V''(1)-2}{4} z_0 z_0^\dagger\right)^{-1/2}\\ \label{eq:MApowerMinus12}
&= \mathbb I + \left(\left(\frac{V''(1)+2}4\right)^{-1/2}-1\right) z_0 z_0^\dagger
= \mathbb I + \left(\frac1{\sqrt{\Delta Q(1)}}-1\right) z_0 z_0^\dagger.
\end{align}
Using that $z z^\dagger$ has rank $1$, one may also derive explicitly that
\begin{align} \label{eq:rotationalMAmeas}
\det \partial\bar\partial\mathscr Q(z) &= \frac{(V'(|z|))^{d-1}\big(|z| V''(|z|)+V'(|z|)\big)}{2^{d+1} |z|^{d}}
\end{align}
which is positive under the conditions that we put on $V$, except possibly in $z=0$.  Hence, $\mathscr Q$ is strictly subharmonic on $\mathbb C^d$, and indeed, strictly plurisubharmonic. 
Note in particular that
\begin{align*}
\det \partial\bar\partial\mathscr Q(z_0) &= \Delta Q(|z_0|), \qquad |z_0|=1.
\end{align*}
Thus, under the conditions we put on $V$, for $|z_0|=1$, the complex Hessian has $d-1$ eigenvalues $1$ and one eigenvalue $\Delta Q(|z_0|)>0$, hence is strictly positive definite.\\

Let us now focus on the corresponding orthogonal polynomials.
In this section we shall use multi-index notation, i.e., if $j=(j_1,\ldots,j_d)\in (\mathbb Z_{\geq 0})^d$, then $z^j = z_1^{j_1} \cdots z_d^{j_d}$ and $j! = j_1! \cdots j_d!$. Furthermore, we denote $|j|=j_1+\ldots+j_d$ (when it is clear that $j$ is to be interpreted as a multi-index).

\begin{lemma} \label{lem:BasisRadial}
    Suppose that $\mathscr Q(z)=V(|z|)$, where $V:[0,\infty)\to\mathbb R$ is a continuous function satisfying \eqref{eq:growthCondV} and
    $$\lim_{r\to 0} r V'(r) = 0.$$
    An orthonormal basis of polynomials with respect to $e^{-n \mathscr Q(z)} d\omega(z)$ is given by
    \[
    \mathscr P_j(z) = \frac1{\sqrt{j! h_{|j|}}} z^j, \qquad j\in (\mathbb Z_{\geq 0})^d,
    \]
    where, for $j=0,1,\ldots$, we have
    \[
    h_{j} = \frac{2}{\Gamma(j+d)} \int_0^\infty r^{2d-1+2j} e^{-n V(r)} dr. 
    \]
\end{lemma}

\begin{proof}
    Any $z\in\mathbb C^d$ can be written as $z=r s$, where $r=|z|$ and $s\in \mathbb S^{2d-1}$. We denote by $d\Omega(s)$ the standard volume form on $\mathbb S^{2d-1}$. With these notations, we may write
    \begin{align*}
    \int_{\mathbb C^d} z^j \overline{z^k} e^{-n\mathscr Q(z)} d\omega(z)
    &= \frac1{\pi^d}\int_0^\infty r^{2d-1+|j|+|k|} e^{-n V(r)} dr \, \int_{\mathbb S^{2d-1}} s^j \overline s^k \, d\Omega(s).
\end{align*}
    We claim that the right-most integral is nonzero if and only if $j=k$ (all multi-index components must match). To prove this claim, let us consider the particular model with $\mathscr Q(z)=|z|^2$, or equivalently, $V(r)=r^2$. In that case it is a known fact that
    \begin{align*}
        \delta_{j,j'} &= \int_{\mathbb C^d} \prod_{k=1}^d n^{j_k+j'_k+1}\frac{z^{j_k} \overline{z^{j'_k}}}{\sqrt{j_k!j'_k!}} e^{-n |z_k|^2} dA(z_{k})\\
        &= n^{\frac{|j|+|j'|}2+d}\int_{\mathbb C^d} \frac{z^j}{\sqrt{j'!}} \frac{\overline{z^k}}{\sqrt{k!}} e^{-n |z|^2} d\omega(z)\\
        &= \frac{n^{\frac{|j|+|j'|}2+d}}{\pi^d\sqrt{j!j'!}} \int_0^\infty r^{|j|+|j'|+2d-1} e^{-n r^2} dr \int_{\mathbb S^{2d-1}} s^j \overline s^k \, d\Omega(s).
    \end{align*}
    Whatever the integral over $r$ is, it has a positive value, therefore the integral over $\mathbb S^{2d-1}$ is nonzero only when $j=k$. In fact, calculating the integral over $r$, we infer that
    \begin{align} \label{eq:BasisRadial}
    \int_{\mathbb S^{2d-1}} s^j \overline s^k \, d\Omega(s)=
    \begin{cases}
    2\pi^d
        \displaystyle\frac{j!}{\Gamma(|j|+d)}, & j=k,\\
        0, & j\neq k.
    \end{cases}
    \end{align}
\end{proof}
We note that an alternative proof of the orthogonality can be found in \cite[Lemma 2.2]{Tian1990}.
 
Using the orthonormal basis given by the monomials, we have
\begin{align} \nonumber
    \mathscr K_n(z,w) &= e^{-\frac12n(V(|z|)+V(|w|))}\sum_{j=0}^{n-1} \sum_{|j'|=j} \frac1{h_j} \frac{z^{j'} \overline w^{j'}}{j'!}\\ \label{eq:kernelDimReduction}
    &= e^{-\frac12 n (V(|z|)+V(|w|))} \sum_{j=0}^{n-1} \frac1{h_j} \frac{(z\cdot w)^j}{j!},
\end{align}
which can be shown by using the generating function $e^{s (z\cdot w)} = e^{s z_1\overline w_1} \cdots e^{s z_d\overline w_d}$.
We can neatly relate this model to the planar model with potential $Q:\mathbb C\to\mathbb R$ given by $Q(z)=V(|z|)$. This planar model has a basis of planar orthogonal polynomials given by
\[
P_j(z) = \frac1{\sqrt{j! h_{j-d+1}}} z^j, \qquad j=0,1,\ldots
\]
In what follows, let us use the notation
\[
\tilde P_j(z) = z^{-d+1} P_{j+d-1}(z).
\]
It follows from \cite[equation (5.8)]{HedenmalmWennman2021} that as $n\to\infty$
\begin{multline} \label{eqtildePbehav}
    \left|\tilde P_j\left(1+\frac{\xi}{\sqrt{n \Delta Q(1)}}\right)\right|^2 e^{-n V\left(1+\frac{\xi}{\sqrt{n \Delta Q(1)}}\right)}\\= 
    \frac1{\sqrt{2\pi}} \sqrt{n \Delta Q(1)} \exp\left(-\frac12\left(2\mathrm{Re} \xi + \frac{n-j}{\sqrt{n\Delta Q(1)}}\right)^2\right)
    \left(1+\mathcal O\left(\frac{\log^3 n}{\sqrt n}\right)\right)
\end{multline}
uniformly for $\xi\in\mathbb C$ with $|\xi|=\mathcal O(\sqrt{\log n})$ and $n-\sqrt n \log n\leq j<n$, where the implied constant can be picked independent of $z_0$.
\cite{HedenmalmWennman2021} considers the case that $Q$ is $1$-admissible, but \eqref{eqtildePbehav} is also true in the less restrictive (though roational symmetric) setting of \cite{AkemannByunEbke2023}, e.g., see the proof of Theorem 1.2 in \cite{AkemannByunEbke2023} where the norm of the $P_j$ is estimated. Hence, \eqref{eqtildePbehav} is valid in the setting of Theorem \ref{thm:generalMain2}.

\begin{proposition} \label{eq:rotationalCheckQ}
    Let $\mathscr Q(z)=V(|z|)$, where $V$ is assumed to be $C^2$ on $(0,\infty)$,  $r\mapsto r V'(r)$ is strictly increasing on $(0,\infty)$, $V'(1)=2$, and
    \[
    \lim_{r\to 0} r V'(r)=0. 
    \] 
    Then we have $\partial S_{\mathscr Q}=\{z\in\mathbb C^d : |z|\leq 1\}$ and $\check{\mathscr Q}(z)=\check Q(|z|)$.
\end{proposition}

\begin{proof}
We first show that $\check{\mathscr Q}(z)=\check Q(|z|)$ is indeed the obstacle function. We clearly have
\[
\check{\mathscr Q}(z)=\check Q(|z|)\leq Q(|z|)=\mathscr Q(|z|),
\]
and as $|z|\to\infty$
\[
\check{\mathscr Q}(z)=\check Q(|z|)\leq \log |z|^2+\mathcal O(1).
\]
Furthermore, since $\check Q$ is subharmonic, it follows that $\check{\mathscr Q}$ is also subharmonic and hence plurisubharmonic. It remains to show that $\check{\mathscr Q}$ is maximal with these properties. Let $q:\mathbb C^d\to [-\infty,\infty)$ be another function with these properties. In our setting $q$ is necessarily rotational symmetric (if not, $q$ would not be unique by permutation of variables). Hence $q(z)=v(|z|)$ for some $v:[0,\infty)\to [-\infty,\infty)$. Let $h:\mathbb C\to[-\infty,\infty)$ be defined by $h(z)=v(|z|)$. We claim that $h$ is subharmonic. Indeed, we see that for any $r>0$
\[
\frac1{2\pi} \int_0^{2\pi} h(r e^{i t}) dt = h(r)\geq h(r).
\]
Hence $q(z)=h(|z|)\leq \check Q(|z|)$. We conclude that $\check{\mathscr Q}(z)=\check Q(|z|)$ for all $z\in\mathbb C^d$.

The conditions we put on $V$ force $S_Q=S_Q^\star$ to be the closed unit disk, hence $S_{\mathscr Q}^\star=\{z\in\mathbb C^d : |z|\leq 1\}$. The condition that $(r V'(r))'>0$ in combination with \eqref{eq:rotationalMAmeas} force the Monge-Ampère measure to be strictly positive on the predroplet, hence $S_{\mathscr Q}=S_{\mathscr Q}^\star$.
\end{proof}

\begin{proposition} \label{prop:rotationalKerDiag}
    Let $\mathscr Q(z)=V(|z|)$, where $V$ is assumed to be $C^2$ on $(0,\infty)$,  $r\mapsto r V'(r)$ is strictly increasing on $(0,\infty)$, $V'(1)=2$, and
    \[
    \lim_{r\to 0} r V'(r)=0. 
    \] 
     We have as $n\to\infty$ that
    \begin{multline*}
        \frac1{\det n\partial\bar\partial \mathscr Q(z_0)}\mathscr K_n\left(z_0+\frac{\xi \vec n(z_0)}{\sqrt{n \det \partial\bar\partial\mathscr Q(z_0)}}\right)\\
        = \frac12\mathrm{erfc}\left(\sqrt 2 \mathrm{Re} \xi\right) \left(1+\mathcal O\left(\frac{\log n}{\sqrt n}\right)\right)
    \end{multline*}
    uniformly for $z_0\in\mathbb C^d$ with $|z_0|=1$ and  $\xi\in\mathbb C$ with $|\xi|=\mathcal O(\sqrt{\log n})$.
\end{proposition}

\begin{proof}
    By radial symmetry, we have $\vec n(z_0)=z_0$. By \eqref{eq:kernelDimReduction} and \eqref{eqtildePbehav}, we have uniformly for $\xi\in\mathbb C$ with $|\xi|=\mathcal O(\sqrt{\log n})$ that
    \begin{multline*}
        \frac1{\det n\partial\bar\partial \mathscr Q(z_0)}\mathscr K_n\left(z_0+\frac{\xi \vec n(z_0)}{\sqrt{n \det \partial\bar\partial\mathscr Q(z_0)}}\right)\\
        = e^{-n V\left(1+\frac{\xi}{\sqrt{n \Delta V(1)}}\right)} \sum_{j=0}^{n-1} \frac{(j+d-1)!}{n^d \sqrt{\Delta Q(1)} j!}\left|\tilde P_{j}\left(1+\frac{\xi}{\sqrt{n \Delta Q(1)}}\right)\right|^2\\
        = \frac1{\sqrt{2\pi}}\sum_{j=0}^{\mathcal O(\sqrt n\log n)} \frac{1+\mathcal O(n^{-1/2}\log n)}{\sqrt{n \Delta Q(1)}} \exp\left(-\frac12\left(2\mathrm{Re} \xi + \frac{j+d-1}{\sqrt{n\Delta Q(1)}}\right)^2\right)\\
        \left(1+\mathcal O\left(\frac{\log^3n}{\sqrt n}\right)\right)
    \end{multline*}
    as $n\to\infty$. By a standard Riemann sum argument, e.g., as in \cite{HedenmalmWennman2021} (or Section \ref{sec:2}), the large $n$ behavior of the sum is given by
    \[
    \frac1{\sqrt{2\pi}}\int_0^\infty e^{-\frac12(2\mathrm{Re} \xi + t)^2} dt
    = \frac12 \mathrm{erfc}(\sqrt 2 \mathrm{Re} \, \xi),
    \]
    up to an $\mathcal O(1/\sqrt n)$ error, uniformly for $|\xi|=\mathcal O(\sqrt{\log n})$.
\end{proof}

\begin{corollary} \label{cor:OutwardXiEta}
 Let $\mathscr Q(z)=V(|z|)$, where $V$ is assumed to be $C^2$, $r\mapsto r V'(r)$ is strictly increasing, and $V'(1)=2$. Let $z_0\in\mathbb C^d$ with $|z_0|=1$. We have as $n\to\infty$ that
    \begin{multline*}
        \frac1{\det n\partial\bar\partial \mathscr Q(z_0)}\mathscr K_n\left(z_0+\frac{\xi \vec n(z_0)}{\sqrt{n \det \partial\bar\partial\mathscr Q(z_0)}}, z_0+\frac{\eta \vec n(z_0)}{\sqrt{n \det \partial\bar\partial\mathscr Q(z_0)}}\right)\\
        \equiv \frac12 \exp\left(\xi \overline \eta-\frac{|\xi|^2+|\eta|^2}{2}\right)\mathrm{erfc}\left(\frac{\xi+\overline\eta}{\sqrt 2}\right)
        \left(1+\mathcal O\left(\frac{\log^3n}{\sqrt n}\right)\right)
    \end{multline*}
    uniformly for $z_0\in\mathbb C^d$ with $|z_0|=1$ and  $\xi, \eta\in\mathbb C$ with $|\xi|=\mathcal O(\sqrt{\log n})$.
\end{corollary}

\begin{proof}
We notice that
\begin{multline*}
\mathscr K_n(z,w) = e^{n(Q(\sqrt{z\cdot w})-\frac12 Q(z)-\frac12 Q(w))}
\sum_{j=0}^{n-1} e^{-n Q(\sqrt{z\cdot w})}|\tilde P_j(\sqrt{z\cdot w})|^2 e^{i j\arg(z\cdot w)}.
\end{multline*}
When 
\begin{align*}
z = z_0+\frac{\xi z_0}{\sqrt{n\det \partial\bar\partial \mathscr Q(z_0)}},
\quad w = z_0+\frac{\eta z_0}{\sqrt{n\det \partial\bar\partial \mathscr Q(z_0)}}
\end{align*}
we have uniformly for $|\xi|,|\eta|=\mathcal O(\sqrt{\log n})$ that
\[
\sqrt{z\cdot w} = 1+\frac{\xi+\overline\eta}{2\sqrt{n\Delta Q(1)}}+\mathcal O\left(\frac{\log n}n\right),
\]
and
\[
\arg(z\cdot w) = \frac1{2\sqrt{n\Delta Q(1)}} \mathrm{Im}(\xi+\overline\eta) + \mathcal O\left(\frac{\log n}{n}\right)
\]
as $n\to\infty$. 
We now follow the same approach as in the proof of Proposition \ref{prop:rotationalKerDiag}, but we replace
\[
-\frac12\left(2\mathrm{Re} \xi + \frac{j+d-1}{\sqrt{n\Delta Q(1)}}\right)^2
\]
in the exponential by
\[
-\frac12\left(\xi+\overline\eta + \frac{j+d-1}{\sqrt{n\Delta Q(1)}}\right)^2
\]
which yields the statement after the standard Riemann sum argument. Note that $e^{i n \arg(z\cdot w)}$ plays the role of the co-cycle. 
\end{proof}

\begin{proposition} \label{prop:rotInvU}
    Let $\mathscr Q(z)=V(|z|)$, where $V$ is assumed to be $C^2$, $r\mapsto r V'(r)$ is strictly increasing, and $V'(1)=2$. Let $z_0\in\mathbb C^d$ with $|z_0|=1$. There exists a unitary matrix $\mathscr U(z_0)$ such that as $n\to\infty$
    \begin{multline*}
        \frac1{\det n \partial\bar\partial\mathscr Q(z_0)} \mathscr K_n\left(z_0+\frac{\mathscr U(z_0)\xi}{\sqrt{n \partial\bar\partial\mathscr Q(z_0)}},z_0+\frac{\mathscr U(z_0)\eta}{\sqrt{n \partial\bar\partial\mathscr Q(z_0)}}\right)    \\
        \equiv
        \frac12 \exp\left(\xi \cdot \eta-\frac{|\xi|^2+|\eta|^2}{2}\right)
        \mathrm{erfc}\left(\sqrt 2 \mathrm{Re} \,\sum_{k=1}^d \frac{\xi_k}{\sqrt {d}}\right)
        \left(1+\mathcal O\left(\frac{\log^3n}{\sqrt n}\right)\right)
    \end{multline*}
    uniformly for $z_0\in\mathbb C^d$ with $|z_0|=1$ and  $\xi\in\mathbb C^d$ with $|\xi|,|\eta|=\mathcal O(\sqrt{\log n})$.
    \end{proposition}

\begin{proof}
    There exists a unitary matrix $\mathscr U(z_0)$ (a rotation in $\mathbb R^{2d}$) such that
\[
\mathscr U(z_0)^\dagger z_0 = \frac1{\sqrt d} (1,1,\ldots,1).
\]
We pick such a matrix henceforth. Using \eqref{eq:MApowerMinus12} we have
\[
\left(\partial \bar\partial\mathscr Q(z_0)\right)^{-1/2} z_0 = \frac{z_0}{\sqrt{\Delta Q(1)}}.
\]
Hence, we have
\[
\mathscr U(z_0)^\dagger \left(\partial \bar\partial\mathscr Q(z_0)\right)^{-1/2} z_0 = \frac1{\sqrt{d\Delta Q(1)}}(1,1,\ldots,1).
\]
Now notice that
\begin{align*}
    \left(z_0+\frac{\mathscr U(z_0)\xi}{\sqrt{n \partial\bar\partial\mathscr Q(z_0)}}\right) &\cdot \left(z_0+\frac{\mathscr U(z_0)\eta}{\sqrt{n \partial\bar\partial\mathscr Q(z_0)}}\right) \\
    &= 
    1+\frac{\xi\cdot (\mathscr U(z_0)^\dagger z_0)}{\sqrt{n \partial\bar\partial\mathscr Q(z_0)}}
    +\frac{(\mathscr U(z_0)^\dagger z_0)\cdot\eta}{\sqrt{n \partial\bar\partial\mathscr Q(z_0)}}
    + \frac{\xi\cdot\eta}{n \partial\bar\partial\mathscr Q(z_0)}\\
    &= 1+\sum_{k=1}^d \frac{\xi_k+\overline{\eta_k}}{\sqrt{d \Delta Q(1)}}
    + \frac{\xi\cdot\eta}{n \partial\bar\partial\mathscr Q(z_0)}\\
    &= \left(1+\frac{\xi_1+\ldots+\xi_d}{\sqrt{d\Delta Q(1)}}\right) \overline{\left(1+\frac{\eta_1+\ldots+\eta_d}{\sqrt{d\Delta Q(1)}}\right)}
    + \mathcal O\left(\frac{\log n}n\right)
\end{align*}
as $n\to\infty$, uniformly for $|\xi|,|\eta|=\mathcal O(\sqrt{\log n})$. 
We may essentially ignore the last term, as it is of negligible order. 
The result now follows by applying the same strategy as in the proof of Proposition \ref{prop:rotationalKerDiag} and Corollary \ref{cor:OutwardXiEta}.
\end{proof}

\begin{proof}[Proof of Theorem \ref{thm:generalMain2}]
We simply collect the results from Corollary \ref{cor:OutwardXiEta} and Proposition \ref{prop:rotInvU}.
\end{proof}

\subsection{An edge scaling limit for counting statistics} \label{sec:3.2}

In this final subsection for the rotational symmetric case, we investigate another type of edge behavior, that of a particular type of linear statistics called counting statistics. While linear statistics in general are global objects, in the rotational symmetric setting they have a local flavor. Given a potential $\mathscr Q:\mathbb C\to \mathbb R$, 
for any test function $f:\mathbb C^d\to \mathbb C$ we may consider the linear statistic
\[
\sum_{j=0}^{N_n^d} f(z_{(j)})
\]
where the summation is over all $N_n^d=\binom{n+d-1}{d}$ points $z_{(j)}\in\mathbb C^d$ of the associated DPP. As is well-known, the variance of this linear statistic is given by
\begin{align} \label{eq:defVar}
\int_{\mathbb C^d} \int_{\mathbb C^d} |f(z)-f(w)|^2 \left|\mathscr K_n(z,w)\right|^2 d\omega(z) d\omega(w).
\end{align}
For test functions $f$ that are Lippschitz with compact support contained in the bulk $\mathring S_{\mathscr Q}$, Berman proved a Central Limit Theorem \cite{Berman2018} (under the assumption that $\mathscr Q$ is locally $C^{1,1}$) and in particular that the limiting variance behaves like
\begin{align} \label{eq:BermanCLT}
\sigma^2 = n^{d-1} \int_{S_{\mathscr Q}} |\nabla f(z)|^2 d\omega(z). 
\end{align}
The situation gets more interesting when we allow the support of $f$ to intersect the droplet boundary $\partial S_{\mathscr Q}$. For $d>1$, it is not known what happens in the general case, but the $d=1$ case is well-understood \cite{AmeurHedenmalmMakarov2011, AmeurHedenmalmMakarov2015}. Ameur, Hedenmalm and Makarov proved that the limiting variance is now $\sigma^2+\tilde\sigma^2$, where $\tilde\sigma^2$ can be expressed using the Neumann jump operator. An analogous limiting formula was proved by Leblé and Serfaty for 2D Coulomb gases with general temperature \cite{LebleSerfaty2016}.  
Such a formula was first proved for the Ginibre ensemble $Q(z)=|z|^2$ by Rider and Virág, in which case we have the particularly appealing form
\[
\tilde\sigma^2 = \frac12 \|f\|_{H^{1/2}(\mathbb S^1)}^2 = \frac12 \sum_{\ell\in\mathbb Z} |\ell| |\hat f(\ell)|^2,
\]
where $\hat f(\ell)$ denotes the $\ell$-th Fourier coefficient
\[
\hat f(\ell) = \frac1{2\pi i} \int_{-\pi}^\pi f(e^{i t}) e^{-i \ell t} dt.
\]
A similar limiting formula was recently proved in a related geometric setting \cite{Ioos2025}.

In this section we focus on radial counting statistics, we let $\mathscr Q$ be a rotational symmetric potential satisfying the conditions of Theorem \ref{thm:generalMain2}, and we let $N_n^d(a)$ be the random variable that gives the number of points in the $2d$-dimensional ball $|z|\leq a$. This corresponds to the choice of (non-smooth) test function
\[
f(z) = \mathfrak{1}_{B(0,a)}(z). 
\]
In this case the variance of the counting statistics is usually called the number variance. 
Since we are interested in edge behaviors, we consider the choice
\begin{align} \label{eq:defandelta}
a=a_n(\delta)=1+\frac{\delta}{\sqrt{2n \Delta Q(1)}},
\end{align}
for $\delta\in\mathbb R$. 
For $d=1$ it was proved by Akemann, Byun and Ebke \cite{AkemannByunEbke2023} (see also \cite{LacroixMajumdarSchehr2019}) that for rotational symmetric potentials
\[
\lim_{n\to\infty} \frac1{\sqrt{n\Delta Q(1)}} \mathrm{Var} \, N_n^d(a_n(\delta))
= \frac1{\sqrt\pi} f(\delta),
\]
where
\[
f(\delta) = \sqrt{2\pi} \int_\delta^\infty \frac{\mathrm{erfc}(t) \mathrm{erfc}(-t)}{4} \, dt.
\]
This was extended to the non-rotational symmetric setting in \cite{AkemannDuitsMolag2026} and \cite{MarzoMolagOrtegaCerda2026}. 
Our goal in this section is to show that a similar limiting formula holds for $d>1$. First, we start with a general lemma. We remind the reader that $P_j$ are the degree $j$ planar orthogonal polynomials with positive leading coefficient such that
\[
\int_{\mathbb C} P_j(z) P_k(z) e^{-n Q(z)} dA(z) = \delta_{j,k}, 
\qquad j,k=0,1,\ldots
\]

\begin{lemma} \label{lem:VarLinStat}
Let $J\in L^p([0,\infty)^2)$ for some $p\geq 1$. Assume that $\mathscr Q:\mathbb C^d\to\mathbb R$ is a rotational symmetric potential that satisfies the conditions of Theorem \ref{thm:generalMain2}. Then there exists an $\epsilon>0$ and a $c>0$ such that as $n\to\infty$ 
\begin{multline} \label{eq:VarInts}
\Gamma(d)\pi^{2d} \int_{\mathbb C^d} \int_{\mathbb C^d} J(|z|, |w|) \left|\mathscr K_n(z,w)\right|^2 d\omega(z) d\omega(w)\\
=\iint\limits_{|z|,\,|w|\leq 1-\epsilon\frac{\log n}{\sqrt{n}}} J(|z|, |w|) \left|\hat K_n(z,w)\right|^2 dA(z) dA(w)\\
 +  n^{d-1}(1+\mathcal O(\frac1{\sqrt n}))\iint\limits_{1-\epsilon \frac{\log n}{\sqrt n}\leq |z|,|w|\leq 1+\epsilon \frac{\log n}{\sqrt n}} J(|z|, |w|) \left|K_{n+d-1}(z,w)\right|^2 dA(z) dA(w)\\
 + \mathcal O(e^{-c\log^2 n}),
\end{multline}
for some constant $c>0$, where 
\begin{align*}
K_n(z,w) &= e^{-\frac12n(Q(z)+Q(w))}\sum_{j=0}^{n-1} P_j(z) \overline{P_j(w)}\\
\hat K_n(z,w) &= e^{-\frac12n(Q(z)+Q(w))}\sum_{j=0}^{n+d-2} \sqrt{\frac{\Gamma(j+1)}{\Gamma(j-d+2)}} P_j(z) \overline{P_j(w)}.
\end{align*}
\end{lemma}

\begin{proof}
As before we write $\mathscr Q(z)=Q(|z|)$ and without loss of generality we set $Q(1)=0$. 
For convenience, we denote $\hat J(|z|,|w|)=e^{-n Q(|z|)} e^{-n Q(|w|)} J(|z|,|w|)$. 
Write $z=r \Omega$ and $w=r'\Omega'$, where $\Omega, \Omega'$ are in the $2d-1$ dimensional unit sphere $\mathbb S^{2d-1}\subset \mathbb C^d$. 
Using the expression for the orthogonal polynomials in Lemma \ref{lem:BasisRadial}, and in particular \eqref{eq:BasisRadial} to go from the second to the third line, we find
\begin{align*}
&\int_{\mathbb C^d} \int_{\mathbb C^d} J(|z|, |w|) \left|\mathscr K_n(z,w)\right|^2 d\omega(z) d\omega(w)\\
&= \frac1{\pi^{2d}} \int_0^\infty \int_0^\infty \hat J(r, r') \int_{\mathbb S^{2d-1}} \int_{\mathbb S^{2d-1}} \sum_{|j|, |j'|<n} \frac{z^j \overline w^j \overline z^{j'} w^{j'}}{j! h_{|j|} j'! h_{|j'|} } d\Omega d\Omega' r^{2d-1} dr r'^{2d-1} dr'\\
&= \frac{2}{\pi^{2d}}\int_0^\infty \int_0^\infty \hat J(r,r') \int_{\mathbb S^{2d-1}} \sum_{|j|<n} \frac{\Omega^j \overline\Omega^j}{j! \Gamma(|j|+d) h_{|j|}^2} (r r')^{2|j|+2d-1} d\Omega dr dr'\\
&= \frac{2}{\pi^{2d}}\int_0^\infty \int_0^\infty \hat J(r,r') \int_{\mathbb S^{2d-1}} \sum_{j=0}^{n-1} \frac{(\Omega\cdot \Omega)^{n-1}}{j! \Gamma(j+d) h_{j}^2} (r r')^{2 j+2d-1} d\Omega dr dr'\\
&= \frac{4}{\pi^{2d}\Gamma(d)}\int_0^\infty \int_0^\infty \hat J(r,r') \sum_{j=0}^{n-1} \frac{1}{j! \Gamma(j+d) h_{j}^2} (r r')^{2 j+2d-1}dr dr'\\
&= \frac{4}{\pi^{2d}\Gamma(d)}\int_0^\infty \int_0^\infty \hat J(r,r') \sum_{j=0}^{n+d-2} \frac{\Gamma(j+1)}{\Gamma(j-d+2)} \frac{1}{j!^2 h_{j-d+1}^2} (r r')^{2 j+1} dr dr'.
\end{align*}
As remarked before, explicitly, we have
\[
P_j(z) = \frac{z^j}{\sqrt{j! h_{j-d+1}}},
\]
with $h_j$ as in Lemma \ref{lem:BasisRadial} and hence
\begin{multline*}
4\int_0^\infty \int_0^\infty F(r,r') \sum_{j=0}^{n-1} \frac{j!}{(j-d+1)!} \frac{1}{j!^2 h_{j-d+1}^2} (r r')^{2 j+1}dr dr'\\
= \int_{\mathbb C} \int_{\mathbb C} f(|z|,|w|) \left|\hat K_n(z,w)\right|^2 dA(z) dA(w).
\end{multline*}
It remains to estimate the integrand in the relevant regions. Note that
\begin{align*}
\left|\hat K_n(z,w)\right|^2
\leq n^{d-1} K_{n+d-1}(|z|,|w|)^2
\end{align*}
for all $z,w\in\mathbb C$. When $|z|\geq 1+\epsilon \frac{\log n}{\sqrt n}$ and $z\in\mathbb C$ we have by Cauchy-Schwarz and a well-known estimate 
\[
K_{n+d-1}(|z|,|w|)^2 \leq K_{n+d-1}(z) K_{n+d-1}(w)\lesssim n^2 e^{-n (Q(z)-\check Q(z))} e^{-n(Q(w)-\check Q(w))}.
\]
We know that $Q-\check Q$ behaves quadratically (e.g., see \cite[Proposition 3.6]{HedenmalmWennman2021}) just outside the droplet. Hence there exists a constant $\lambda>0$ (independent of $z$) such that
\[
Q(|z|)-\check Q(|z|) \geq \frac12(Q(|z|)-\check Q(z))+\lambda (|z|-1)^2 \geq \frac12(Q(|z|)-\check Q(z))+\lambda \epsilon^2 \frac{\log^2n}{n}.
\]
Combined with the growth conditions \eqref{eq:growthQ} and \eqref{eq:growthCheckQ} on $Q$ and $\check Q$ for large $|z|$, this shows us that this contribution to the integral is of order $e^{-c\log^2n}$ for some constant $c>0$. Next, assume that $|z|\leq 1-\epsilon \frac{\log n}{\sqrt n}$ while $1-\epsilon \frac{\log n}{\sqrt n}\leq |w|\leq 1+\epsilon \frac{\log n}{\sqrt n}$. By the inequality in \cite[Corollary 8.2]{AmeurHedenmalmMakarov2010} we have
\[
K_{n+d-1}(|z|,|w|)^2  \lesssim  n^{2} e^{-\lambda_0 \sqrt n \min(||z|-1|,||z|-|w||)} e^{-n(Q(w)-\check Q(w))},
\]
where $\lambda_0>0$ and the implied constant is independent of $z$ and $w$. We infer that
\[
n^{d-1} K_{n+d-1}(|z|,|w|)^2  \lesssim  n^{d+1-\lambda_0 \epsilon},
\]
which is small for $\epsilon>0$ big enough. Then there are two integration regions that remain. For the region $|z|,|w|\leq 1-\epsilon\frac{\log n}{\sqrt n}$ there is nothing left to prove. For the remaining region where $1-\epsilon \frac{\log n}{\sqrt n}\leq |z|,|w|\leq 1+\epsilon \frac{\log n}{\sqrt n}$, since indices $j\leq n-\sqrt n \log n$ do not contribute to the dominant order we have
\begin{align*}
\hat K_n(|z|,|w|) &= e^{-\frac12n(Q(|z|)+Q(|w|))} \sum_{j=\lceil n-\sqrt n\log n\rceil}^{n+d-2} n^{d-1}(1+\mathcal O(\frac{\log n}{\sqrt n})) P_j(|z|) \overline{P_j(|w|)}\\
&= n^{d-1}(1+\mathcal O(\frac{\log n}{\sqrt n})) e^{\frac12(d-1) (Q(|z|)+Q(|w|))} K_{n+d}(|z|,|w|)\\
&= n^{d-1}(1+\mathcal O(\frac{\log n}{\sqrt n})) K_{n+d}(|z|,|w|).
\end{align*}
\end{proof}

Note that $K_n$ is simply the correlation kernel for the planar weight $Q:\mathbb C\to \mathbb R$. However, it is not implied that $\hat K_n$ necessarily has the interpretation of a correlation kernel. Though, near the droplet boundary for $z$ and $w$ not too close to each other, it should approximate the Sz\H{e}go kernel \cite{AmeurCronvall2022}. For radial linear statistics we take
\[
J(|z|,|w|) = |f(|z|)-f(|w|)|^2.
\]
The first term in \eqref{eq:VarInts} will yield the term \eqref{eq:BermanCLT} found by Berman (for Lipschitz test functions). The second term, due to \cite{AmeurHedenmalmMakarov2015}, should give the extra term $\tilde\sigma^2$ determined by the Neumann jump operator when $d=1$, but, somewhat anticlimactically, for radial functions $f$ this term $\tilde\sigma^2$ vanishes. We thus cannot extract meaningful information about the general variance term associated to the edge. However, we can say something about the number variance near the edge, that is when we consider a microscopic dilation of the droplet.


\begin{theorem} \label{thm:edgeScalingCounting}
Assume that $\mathscr Q:\mathbb C^d\to\mathbb R$ is a rotational symmetric potential. Let $N_n^d(a)$ denote the number of points in the disc $|z|\leq a$. Then, with $a_n(\delta)$ as defined in \eqref{eq:defandelta}, we have as $n\to\infty$ that
\begin{align*}
\lim_{n\to\infty} \frac1{n^{d-1}\sqrt n} \mathrm{Var} \, N_n^d(a_n(\delta)) =  \frac{f(\delta)}{2\pi \sqrt\pi} \sqrt{\det \partial \bar\partial \mathscr Q(z_0)} \, |\mathbb S^{2d-1}|,
\end{align*}
uniformly for $\delta\in\mathbb R$ in compact sets, any $z_0\in\mathbb C^d$ with $|z_0|=1$, and
\[
f(\delta) = \sqrt{2\pi} \int_\delta^\infty \frac{\mathrm{erfc}(t) \mathrm{erfc}(-t)}{4} \, dt.
\]
\end{theorem}

\begin{proof}
Now we consider 
\[
J(|z|,|w|) = (\mathfrak{1}_{[0,a_n(\delta)]}(|z|)-\mathfrak{1}_{[0,a_n(\delta)]}(|w|))^2.
\]
Note that the first integral on the right-hand side of \eqref{eq:VarInts} is $0$ in this case. 
By Lemma \ref{lem:VarLinStat}, first for $d\geq 1$ and then for $d=1$, we have
\begin{multline*}
\frac1{n^{d-1}} \pi^{2d} \Gamma(d) \mathrm{Var} \, N_n^d(a_n(\delta))\\
= (1+\mathcal O(\frac1{\sqrt n}))\iint\limits_{1-\epsilon \frac{\log n}{\sqrt n}\leq |z|,|w|\leq 1+\epsilon \frac{\log n}{\sqrt n}} J(|z|, |w|) \left|K_{n+d-1}(z,w)\right|^2 dA(z) dA(w)\\
 + \mathcal O(e^{-c\log^2 n})\\
 = \pi^2 \mathrm{Var} \, N_n^1(a_n(\delta))
 + \mathcal O(e^{-c\log^2 n})
\end{multline*}
as $n\to\infty$, for some $c>0$. But for $d=1$, the result is already known \cite{AkemannByunEbke2023}. 
\end{proof}

Note that by \cite{MarzoMolagOrtegaCerda2026}, we know that the error is at most of order $1/\sqrt{\log n}$. Moreover, for general potentials, based on Theorem 1.2 in \cite{MarzoMolagOrtegaCerda2026}, one would expect with a suitable microscopic dilation of the droplet to find the limit
\[
\frac{f(\delta)}{2\pi \sqrt\pi} \int_{\partial S_{\mathscr Q}} \sqrt{\det\partial\bar\partial \mathscr Q(z)} d\psi_Q(z),
\]
for some measure $d\psi_Q(z)$. For $d=1$ this measure is the Harmonic measure at $\infty$, it will be interesting to find out what it needs to be replaced by when $d>1$. To extend such results to general potentials (not necessarily rotational symmetric), one would need to understand the kernel asymptotics near the edge but off-diagonally, and obtain a result similar to \cite{AmeurCronvall2022}, which we intend to investigate in a future work.\\


Finally, we briefly comment on the relation between the number variance and entanglement entropy in certain models (in the edge case). 
For $d$=1 the Ginibre ensemble $Q(z)=|z|^2$ has a quantum mechanical interpretation, it describes the locations of noninteracting Fermions in a rotating trap in two dimensions, with repulsion caused by the Pauli exclusion principle. In \cite{LacroixMajumdarSchehr2019} it was shown that, given a concentric disk, the associated number variance and entanglement entropy scale proportionally. This model is easily generalized to $d>1$ by taking a tensor product of (planar) Hamiltonians, corresponding to $\mathscr Q(z)=|z|^2$, which is a model of noninteracting Fermions in $\mathbb C^d$ in a rotating (hypersurface) trap. Alternatively, it may be interpreted as a quantum Hall fluid. We can follow the same approach as in \cite{LacroixMajumdarSchehr2019}. 
The entanglement entropy (or Rényi entropy) corresponding to a ball $\{z\in\mathbb C^d : |z|\leq a\}$ can be expressed as
\[
S_q(n,a) = \frac1{1-q} \mathrm{Tr} \, \log\left(\mathbb A^q+(\mathbb I-\mathbb A)^q\right),
\]
where $q>1$ and $\mathbb A$ denotes the overlap matrix, which is the $N_n^d\times N_n^d$ matrix with components
\[
\mathbb A_{j,j'} = \int_{|z|\leq a} \mathscr P_j(z) \overline{\mathscr P_{j'}(z)} e^{-n \mathscr Q(z)} \, d\omega(z), \qquad j,j'\in J_n.
\]
As $q\to 1^+$ one obtains the von Neumann entropy.
By Lemma \ref{lem:BasisRadial}, this matrix is in fact diagonal and its eigenvalues are given by
\[
\lambda_j(a) = 2\frac{n^{|j|+d}}{\Gamma(|j|+d)} \int_0^a e^{-n r^2} \, r^{2|j|+2d-1} \, dr
= \frac{\gamma(|j|+d, n a^2)}{\Gamma(|j|+d)}, \quad j\in J_n.
\]
By formula 8.12.18 in the NIST handbook \cite{NIST:DLMF} we know that the regularized lower incomplete gamma function satisfies
\begin{align*}
\lambda_j(a_n(\delta)) \sim 
\begin{cases}
\displaystyle\frac12 \mathrm{erfc}\left(\frac{|j|-n a_n(\delta)^2}{\sqrt{2n} a_n(\delta)}\right), & j\leq n a_n(\delta)^2\\
\displaystyle 1- \frac12 \mathrm{erfc}\left(\frac{|j|-n a_n(\delta)^2}{\sqrt{2n} a_n(\delta)}\right), & j\geq n a_n(\delta)^2
\end{cases}
\end{align*}
for $j=n+\mathcal O(\sqrt{n\log n})$ as $n\to\infty$. One may argue that only indices in this range contribute to the dominant order, and then a Riemann sum argument gives us that 
\begin{align*}
S_q(n,a_n(\delta)) &=\frac1{1-q}\sum_{|j|<n} \log\left(\lambda_j(a_n(\delta))^q+(1-\lambda_j(a_n(\delta)))^q\right)\\
&\sim \frac1{1-q}\sum_{j=n-\sqrt{n\log n}}^{n-1} \binom{j+d-1}{d-1} \log\left(\lambda_j(a_n(\delta))^q+(1-\lambda_j(a_n(\delta)))^q\right)\\
&\sim \frac{n^{d-1}\sqrt n}{(d-1)!} \frac{1}{1-q} \int_{\delta}^\infty \log\left(\frac1{2^q}\mathrm{erfc}(-x)^q+\frac1{2^q}\mathrm{erfc}(x)^q\right) \, dx
\end{align*}
as $n\to\infty$. Hence, as in the $d=1$ case in \cite{LacroixMajumdarSchehr2019}, the number variance (see Theorem \ref{thm:edgeScalingCounting} above) and entanglement entropy scale proportionally in the edge regime. The significance of this observation is that entanglement entropy, which is hard to measure directly, can be measured indirectly via the number variance.

\section{Edge point bulk degeneracy: kernels with $o(n)$ terms} \label{sec:4}

As explained in the introduction, some regular edge points $z_0\in \partial S_{\mathscr Q}$ show a certain type of bulk degeneracy. One or more of their coordinates behave as though they are part of the bulk. This is especially explicit in the proof of Lemma \ref{lem:bulkDegen}. 
It is easy to see that, besides the factorized setting in Section \ref{sec:2}, there are many other ways to construct models which exhibit a similar behavior. In order to eventually prove Conjecture \ref{con:1} and Conjecture \ref{con:2}, it appears that we need to understand this bulk degeneracy phenomenon. To treat such edge points, we have to understand the local behavior of the planar kernels. Given a planar potential $Q:\mathbb C\to\mathbb R$ (satisfying \eqref{eq:growthQ}), with associated $n$-dependent planar orthogonal polynomials $P_j$ (of degree $j$ and positive leading coefficient) that satisfy
\[
\langle P_j, P_k\rangle = \int_{\mathbb C} P_j(z) \overline{P_k(z)} e^{-n Q(z)} dA(z) = \delta_{jk}, \qquad j,k=0,1,\ldots,
\]
what we need to understand is the partial kernel
\begin{align*}
\sum_{j=0}^{m_n} P_j(z) \overline{P_j(w)},
\end{align*}
where $m_n$ grows slower than $n$. In our specific case we have that $m_n$ grows like $\sqrt n\log n$, and we only need to understand the partial kernel on the diagonal $\xi=\eta$. Nevertheless, it is hardly any extra work to consider the more general case where $m_n=o(n)$ and not necessarily $\xi=\eta$. 

As mentioned in the introduction, there is a standard approach to derive such results using Hörmander's $\bar\partial$-method, but we will devise a different approach, that seemingly gives us more information. It starts with the well-known fact that the unweighted kernel satisfies the following pointwise extremal property.
\begin{align} \label{eq:partialKerSup}
\sum_{j=0}^{m_n} |P_j(z)|^2 &= \sup_{p\in \mathcal H_{m_n}\setminus\{0\}} \frac{|p(z)|^2}{\int_{\mathbb C} |p(w)|^2 e^{-n Q(w)} dA(w)}
\end{align}
where $\mathcal H_{m_n}$ denotes the space of all polynomials of degree $\leq m_n$.
We start with an off-diagonal decay lemma for the inner products.
We shall use the following notation for the monomials 
\[
e_j(z) = z^j, \qquad j=0,1,\ldots
\]

\begin{lemma} \label{lem:IjkEst}
Let $Q:\mathbb C\to\mathbb R$ be a real-analytic function with a unique minimum at $z=0$, satisfying the growth condition
\begin{align*}
\liminf_{|z|\to\infty} \frac{Q(z)}{\log |z|^2}>1+\epsilon
\end{align*}
for some fixed $\epsilon>0$. 
Then there are constants $0<C_1< 1< C_2$, depending only on $Q$ and $\epsilon$, such that
\[
\left(\frac{C_1}{\Delta Q(0)}\frac{j+k}{2n}\right)^{|j-k|}\leq 
\left|\frac{\langle e_j, e_k\rangle}{\langle e_j, e_j\rangle}\right| \leq \left(\frac{C_2}{\Delta Q(0)}\frac{j+k}{2n}\right)^{|j-k|}
\]
uniformly for nonnegative integers $j, k$ such that $0\leq j+k\leq 2(1+\epsilon) n$.\\
Furthermore, we have for all $0< j \leq (1+\epsilon) n$ that
\[
\left(\frac{C_1}{\Delta Q(0)}\frac{j}{n}\right)^{j}\leq 
\left|\frac{\langle e_j, e_j\rangle}{\langle e_0, e_0\rangle}\right| \leq \left(\frac{C_2}{\Delta Q(0)}\frac{j}{n}\right)^{j}
\]
\end{lemma}

\begin{proof}
We may assume without loss of generality that $Q(0)=0$ and $\Delta Q(0)=1$. Necessarily, the first derivatives of $Q$ vanish at $0$. There exists an $\epsilon'>\epsilon$ such that
\begin{align*}
\liminf_{|z|\to\infty} \frac{Q(z)}{\log |z|^2}\geq 1+\epsilon'.
\end{align*}
Then for some $R>0$ we have
\[
\int_{|z|\geq R} |z|^{j+k} e^{-n Q(z)} dA(z)
\leq \int_{|z|\geq R} \frac{dA(z)}{|z|^{2n (\epsilon'-\epsilon)}}
= \frac{R^{-2n(\epsilon'-\epsilon)+2}}{n(\epsilon'-\epsilon)-2} 
\]
when $n$ is big enough. 
Now we use Bochner normal coordinates, i.e., on a small enough neighborhood, there exists a (bi)holomorphic map $h$ such that
\[
Q(h(z)) = |z|^2 + f(z),
\]
where $f$ is a real-analytic function such that $f(z)=\mathcal O(|z|^4)$ (and furthermore, in the expansion there are no holomorphic powers of $z$ of order $\geq 2$) \cite{Bochner1947}. 
Now consider the map $\psi$ defined as the inverse of the map
\[
z \mapsto z \sqrt{1+f(z)/|z|^2}.
\]
On a small enough neighborhood, one may check that this map is diffeomorphic. So let us divide the remaining integration region into a region $(h\circ \psi)(B(0,r))$ and the region $B(0,R)\setminus (h\circ \psi)(B(0,r))$, where $r>0$ is picked small enough. Since $f$ has a unique minimum, the contribution on the latter region will be exponentially small. 
We conclude that $\langle e_j,e_k\rangle = I_{jk}+\mathcal O(e^{- \eta n})$ for some constant $\eta>0$ that depends only on $Q$, where
\begin{align*}
I_{jk} = \int_{B(0,r)} (h\circ \psi)(z)^j \overline{(h\circ \psi)(z)}^k e^{-n |z|^2} |h'(\psi(z))|^2 |\det D\psi(z)| dA(z).
\end{align*}
Notice in particular that $(h\circ \psi(z))^j = z^j (1+g(z))^j$ 
for some real-analytic function $g$ with $g(0)=0$, and
\[
h'(\psi(z)) \det \psi(z) = \det \psi(0) + \tilde g(z)
\]
for some real-analytic function $\tilde g$ with $\tilde g(0)=0$. 
Suppose that $j\leq k$ and write $m=\frac{j+k}{2}$. For some constants $\lambda>0$ and $C>0$ that depend only on $Q$ we have
\begin{align*}
|I_{jk}| \leq C \int_{\mathbb C} |z|^{2m} e^{-n |z|^2+2\lambda m |z|} dA(z)
= 2 C \left(\frac{m}{n}\right)^{m+1} \tilde I_m(\alpha),
\end{align*}
where 
\begin{align*}
\tilde I_m(\alpha) = \int_{-\alpha}^\infty e^{-m f_\alpha(r)} (r+\alpha) dr,
\qquad f_\alpha(r) = r^2 - 2\log(r+\alpha),
\end{align*}
in our case with the explicit choice
\[
\alpha = \alpha_{m,n} =\lambda \sqrt\frac{m}{n}.
\]
This follows by a combination of translation and rescaling of the integration variables. 
Next we apply Laplace's method for $\alpha\in\mathbb R$ in compact sets, and $m\to\infty$. 
The saddle point function has a unique minimum at $r_+(\alpha) = \sqrt{1+\frac14\alpha^2}-\frac12\alpha$ and Laplace's method yields
\begin{align*}
\int_{-\alpha}^\infty e^{-m f_\alpha(r)} (r+\alpha) dr
= \sqrt{\frac{\pi}{m}} \frac1{(1+\frac14\alpha^2)^{1/4}} e^{-m f_\alpha(r_+(\alpha))} (r_+(\alpha)+\alpha+\mathcal O(1/m))
\end{align*}
as $m\to\infty$, where, with a little care, one can show that the convergence is uniform for $\alpha$ in compact sets. 
Using in particular the estimate
\[
\alpha_{m+1,n}-\alpha_{m,n} \leq \frac{\lambda}{4\sqrt{m n}}.
\]
one derives that
\[
\frac{\tilde I_{m+1}(\alpha_{m+1,n})}{\tilde I_m(\alpha_{m,n})} = 1+\mathcal O\left(\sqrt\frac{m}{n}\right)
\]
as $m\to\infty$, uniformly for $m\leq 2(1+\epsilon)n$. 
In particular, there exist constants $0<C_1\leq 1\leq C_2<1$ such that uniformly for all nonnegative integers $m$ we have
\[
C_1 \leq \left|\frac{\tilde I_{m+1}(\alpha_{m+1,n})}{\tilde I_{m}(\alpha_{m,n})}\right|\leq C_2.
\]
On the other hand, when $j=k$ there exists a constant $c>0$ (depending only on $Q$) such that
\[
I_{jj} \geq c\int_{\mathbb C} |z|^{2j} e^{-n|z|^2-\epsilon|z|} dA(z)
= 2 c \left(\frac{m}{n}\right)^{2j+1} \tilde I_{j},
\]
where we possibly pick $\epsilon>0$ larger. Now assume that $k>j$. Then we have by the above
\[
|I_{jk}| \leq 2 C \left(\frac{m}{n}\right)^{m+1} \left(C_2\right)^{k-j} \tilde I_j(\alpha)
\leq \frac{C}{c} \left(C_2 \frac{m}{n}\right)^{k-j} I_{jj}.
\]
We extend this by symmetry and obtain
\[
\frac{c}{C}\left(C_1 \frac{m}{n}\right)^{|j-k|} \leq \left|\frac{I_{jk}}{I_{jj}}\right| \leq \frac{C}{c} \left(C_2 \frac{m}{n}\right)^{|j-k|},
\]
for all nonnegative integers $j, k$ such that $j+k\leq (1+\epsilon) n$. By picking $C_2$ slightly larger and $C_1$ slightly smaller one may effectively set $c/C=1$. Since the difference between the $I_{jk}$ and $\langle e_j,e_k\rangle$ is exponentially small as $n\to\infty$, we find the stated estimates. 
\end{proof}


Lemma \ref{lem:IjkEst} has a crucial consequence, which will become clear in the proof of the following proposition. We would like to point out that here the advantage with respect to Hörmander's method is apparant, we obtain an asymptotic formula that is uniform for $z\in\mathbb C$. Somewhat surprisingly, after applying Lemma \ref{lem:IjkEst}, the result follows simply from the Lagrange multiplier method. 

\begin{proposition} \label{prop:planarKernelApprox}
Let $Q:\mathbb C\to\mathbb R$ be a real-analytic function with a unique minimum at $z=0$. Then there exists a $\lambda>0$ such that 
\[
\frac1{n\Delta Q(0)} \sum_{j=0}^{m_n} |P_j(z)|^2
= \left(1+\mathcal O\left(\frac{m_n}{n}\right)\right) \sum_{j=0}^{m_n} \frac{\langle e_0,e_0 \rangle}{\langle e_j,e_j \rangle} |z|^{2j}
\]
uniformly for all $z\in\mathbb C$ as $n\to\infty$, under the condition $0\leq m_n\leq \lambda n$.
\end{proposition}

\begin{proof}
We will estimate the expressions in the supremum in \eqref{eq:partialKerSup}.
Write $p(z)=a_{m_n} z^{m_n}+ \ldots a_1 z + a_0$ for arbitrary complex coefficients. Let us also write $J_{jk}=\langle e_j,e_k\rangle$. Consider the $(m_n+1)\times (m_n+1)$ matrix
\[
A_{jk} = \begin{cases}
|a|^j \sqrt{|J_{jk}|}, & k>j,\\
0, & k\leq j.
\end{cases}
\]
Lemma \ref{lem:IjkEst} gives us that
\begin{align*}
\sum_{j=0}^{m_n} |a|^j \sum_{k= j+1}^{m_n} \left|\overline{a_k} J_{jk}\right|
= \mathrm{Tr}(A^\dagger A)= \|A\|^2
&\leq \sum_{j=0}^{m_n-1} \sum_{k={j+1}}^{m_n} |a_j|^2 (C_2 \frac{m_n}{n})^{k-j} |J_{jj}|\\
&\leq C_2 \frac{m_n}{n} \sum_{j=0}^{m_n} |a_j|^2 |J_{jj}|,
\end{align*}
where we make the assumption here and henceforth that $2 C_2 m_n\leq n$.
This estimate shows us that the norm of $p(z)$ is dominated by the diagonal terms, that is
\begin{align*}
\int_{\mathbb C} \left|\sum_{j=0}^{m_n} a_j w^j\right|^2 e^{-n Q(w)} dA(w)
&= \sum_{j=0}^{m_n} |a_j|^2 J_{jj} + 2 \mathrm{Re}\left( \sum_{j=0}^{m_n-1} a_j \sum_{k=j+1}^{m_n} \overline{a_k} J_{jk}\right)\\
&= \left(1+\mathcal O\left(\frac{m_n}{n}\right)\right) \sum_{j=0}^{m_n} |a_j|^2 J_{jj}
\end{align*}
where the implied constant $C_2$ depends only on $Q$. Hence, for any polynomial $p\in\mathcal H_{m_n}\setminus\{0\}$ 
\begin{align*}
\left(1-C_2 \frac{m_n}{n}\right)^{-1} \frac{\left|\sum_{j=0}^{m_n} a_j z^j\right|^2}{\sum_{j=0}^n |a_j|^2 J_{jj}}
&\leq
\frac{|p(\xi)|^2}{\displaystyle \int_{\mathbb C} |p(z)|^2 e^{-n Q(z)} dA(z)}\\
&\leq \left(1+C_2 \frac{m_n}{n}\right)^{-1} \frac{\left|\sum_{j=0}^{m_n} a_j z^j\right|^2}{\sum_{j=0}^n |a_j|^2 J_{jj}} 
\end{align*}
We can determine the maximum of the function appearing in the bounds simply by applying the Lagrange multiplier method, i.e., for fixed $z\in\mathbb C$ we will maximize the function
\[
|p(z)|^2 = \left|\sum_{j=0}^{m_n} a_j z^j\right|^2
= \sum_{j,k=0}^{m_n} a_j \overline{a_k} z^j \overline{z^k}
\]
over $a=(a_0, \ldots, a_{m_n})\in\mathbb C^{m_n+1}$ under the constraint
\[
\sum_{j=0}^{m_n} |a_j|^2 J_{jj} = 1.
\]
Then we find for what $\lambda\in\mathbb R$ there is a solution to the Lagrange multiplier equations
\[
\overline z^k p(z) = \lambda a_k J_{kk}.
\]
We may exclude the case $\lambda=0$, since the kernel is strictly positive on the diagonal.
From this equation we extract that
\[
a_k = a_0 \frac{\overline z^k}{J_{kk}}, \qquad k=0,\ldots,m_n.
\]
Putting this back in the constraint yields 
\[
|a_0|^2 \sum_{k=0}^{m_n} \frac{|z|^{2k}}{J_{kk}} = 1.
\]
Assuming without loss of generality that $a_0>0$, the previous equation gives us
\[
\lambda = \frac{p(z)}{a_0 J_{00}} = \frac{p(z)}{J_{00}} \sqrt{\sum_{k=0}^{m_n} \frac{|z|^{2k}}{J_{kk}}}.
\]
Finally then, the maximum value is given by
\[
|p(z)|^2 = \left|\sum_{j=0}^{m_n} \frac{\overline z^j p(z)}{\lambda J_{jj}} z^j\right|^2
= \sum_{k=0}^{m_n} \frac{J_{00}^2}{J_{kk}} |z|^{2k}.
\]
Indeed, we have as $n\to\infty$ that
\[
J_{00} = \frac1{n\Delta Q(0)} \left(1+\mathcal O\left(\frac{1}{n}\right)\right)
= \frac1{n\Delta Q(0)} \left(1+\mathcal O\left(\frac{m_n}{n}\right)\right).
\]
The result now follows from the extremal property \eqref{eq:partialKerSup}.
\end{proof}

\pagebreak

\begin{proof}[Proof of Theorem \ref{thm:PartialBerg}.]
We may assume without loss of generality that $m_n=o(n)$, since the case where $m_n$ grows proportionally to $n$ is already known (e.g., see \cite{AmeurHedenmalmMakarov2010}). By Lemma \ref{lem:IjkEst} we have
\[
\sum_{j=0}^{m_n} \frac{\langle e_0,e_0 \rangle}{\langle e_j,e_j \rangle} \left|z\right|^{2j} \geq 1+\sum_{j=1}^{m_n} \left(\frac{n \Delta Q(0)}{j C_1}\right)^j \left|z\right|^{2j}
> 1+ \sum_{j=1}^{m_n} \frac{(m_n)^j}{j!}  \left|\frac{n}{m_n} \frac{\Delta Q(0)}{C_1 e} z^2\right|^{j}
\]
where we used the inequality $j!\geq j^j e^{-j+1}>j^j e^{-j}$. It is a well-known fact that
\[
e^{-x}\sum_{j=1}^{m_n} \frac{1}{j!} (m_n)^j x^j
\geq 1 - C \frac{x^{m_n+1}}{(m_n+1)!}
\]
for some constant $C>0$ uniformly for $x$ in compact subsets of $[0,1)$. This is in particular satisfied when $|z|\leq r_Q \sqrt{\frac{m_n}{n}}$ when we pick $r_Q = \sqrt\frac{e}{\Delta Q(0)}$. 
We obviously also have the usual upper bound given by $1$ and we conclude that uniformly for $|z|\leq r_Q$
\[
\lim_{n\to\infty} \frac{e^{-n Q\left(\frac{\sqrt{m_n} z}{\sqrt{n \Delta Q(0)}}\right)}}{n\Delta Q(0)}  \sum_{j=0}^{m_n} \left|P_j\left(\frac{\sqrt{m_n} z}{\sqrt{n\Delta Q(0)}}\right)\right|^2 = 1.
\]
For the second part of the theorem, note that
\[
f_n(z, w) = e^{-n Q\left(\frac{\sqrt{m_n} z}{\sqrt{n\Delta Q(0)}}, \frac{\sqrt{m_n} \, \overline w}{\sqrt{n\Delta Q(0)}}\right)}
\sum_{j=0}^{m_n} P_j\left(\frac{\sqrt{m_n} z}{\sqrt{n\Delta Q(0)}}\right) \overline{P_j\left(\frac{\sqrt{m_n} w}{\sqrt{n\Delta Q(0)}}\right)}
\]
where $Q(\cdot, \cdot)$ denotes the polarization of $Q(\cdot)$, 
converges uniformly to $1$ on the diagonal $z=w$ inside $|z|\leq r_Q$. By a standard polarization argument the convergence then also holds locally uniformly on a neighborhood of the diagonal. Expanding $Q(z,\overline w)-\frac12 Q(z) -\frac12 Q(w)$ in the present scaling, and assuming $m_n=o(n^{2/3})$, one arrives at the result. 
\end{proof}

\vspace{3cm}

\subsection*{Acknowledgments}
\text{ }\\
The author was supported
by the UC3M grant 2024/00002/007/001/023 “Local and global limits of complex-dimensional
DPPs” and is currently supported by the grant PID2024-155133NB-I00, “Orthogonality, Approximation, and Integrability:
Applications in Classical and Quantum Stochastic Processes (ORTH-CQ)” by the Agencia Estatal
de Investigación.
The author thanks KU Leuven for its hospitality during a research visit, and thanks Aron Wennman for several valuable insights, in particular, which orthogonal polynomials should add to the dominant order in the factorized setting. 

\newpage

\begin{appendices}

\section{Appendix: a Gaussian integral identity} \label{sec:A}

\begin{lemma} \label{lem:GaussianErfc}
Let $A$ be a real $d\times d$ symmetric strictly positive definite matrix and let $v\in\mathbb R^d\setminus\{0\}$ and $b\in\mathbb R^d$. Then we have
\begin{multline}
\int_{x\cdot v\geq 0} \exp\left(-\frac12 x\cdot A^{-1} x-b \cdot x\right) d^dx\\
=  \frac12\sqrt{\det 2\pi A}\exp\left(\frac12 b\cdot A b\right)
\mathrm{erfc}\left(\frac{\,\, b\cdot A v}{\sqrt{2 \, v\cdot Av}}\right).
\end{multline}
\end{lemma}

\begin{proof}
Since we can diagonalize $A$ by an orthogonal matrix, we may assume without loss of generality that $A$ is diagonal, say with eigenvalues $a_1, \ldots, a_d$. Without loss of generality we will assume $v_d\neq 0$. We rewrite the integral as 
\begin{align*}
\int_{x\cdot v\geq 0} & \exp\left(-\frac12 x\cdot A^{-1} x-b \cdot x\right) d^dx\\
&= \int_{\mathbb R^{d-1}} \int_{-\frac1{v_d} \sum_{k=1}^{d-1} v_k x_k}^\infty \exp\left(-\frac12 \sum_{k=1}^d \left(a_k^{-1} x_k^2+2 b_k x_k\right)\right) d^dx\\
&= \exp\left(\frac12\sum_{k=1}^d a_k b_k^2\right)
\int_{\mathbb R^{d-1}} \int_{\substack{\frac1{v_d}\sum_{k=1}^d b_k a_k v_k\\ - \frac1{v_d} \sum_{k=1}^{d-1} v_k x_k}}^\infty \exp\left(-\frac12 \sum_{k=1}^d a_k^{-1} x_k^2\right) d^dx\\
&= \sqrt{\det A}\exp\left(\frac12\sum_{k=1}^d a_k b_k^2\right) \\
&\quad \int_{\mathbb R^{d-1}} \int_{\frac1{v_d\sqrt{a_d}}\sum_{k=1}^d b_k a_k v_k}^\infty \exp\left(-\frac12 \sum_{k=1}^{d-1} x_k^2-\frac12 \left(x_d-\sum_{k=1}^{d-1} \frac{v_k}{v_d}\sqrt{\frac{a_k}{a_d}} x_k\right)^2\right) d^dx\\
&= \sqrt{\det A}\exp\left(\frac12\sum_{k=1}^d a_k b_k^2\right) \sqrt{\frac{\pi}2} \\
&\quad  \int_{\mathbb R^{d-1}} \exp\left(-\frac12\sum_{k=1}^{d-1} x_k^2\right) \mathrm{erfc}\left(\frac1{\sqrt 2}\sum_{k=1}^d \frac{b_k a_k v_k}{v_d\sqrt{a_d}}-\frac1{\sqrt 2}\sum_{k=1}^{d-1} \frac{v_k}{v_d} \sqrt{\frac{a_k}{a_d}} x_k\right) d^{d-1}x.
\end{align*}
Each vector in $x\in\mathbb R^{d-1}$ can be written in a unique way as $x=(x\cdot e_0) e_0+ \ldots + (x\cdot e_{d-1}) e_{d-1}$, where we define
\[
e_0 = \frac1{\sqrt{\sum_{k=1}^{d-1} a_k v_k^2}} \begin{pmatrix}
\sqrt{a_1} v_1\\ \vdots \\ \sqrt{a_{d-1}} v_{d-1},
\end{pmatrix}
\]
and $\{e_1, \ldots, e_{d-2}\}$ is any orthonormal basis for the orthogonal complement of $\{e_0\}$. 
Now consider the change of variables $y_k = e_k \cdot x$ where $k=0,\ldots,d-2$. Note that $|x|^2=y_0^2+\ldots+y_{d-2}^2$ since $\{e_0,\ldots, e_{d-2}\}$ forms an orthonormal basis. Note that the Jacobian matrix of our transformation is orthogonal, hence has determinant $1$. We find that
\begin{align*}
\int_{\mathbb R^{d-1}}& \exp\left(-\frac12\sum_{k=1}^{d-1} x_k^2\right) \mathrm{erfc}\left(\frac1{\sqrt 2}\sum_{k=1}^{d-1} \frac{b_k a_k v_k}{v_d\sqrt{a_d}}-\frac1{\sqrt 2}\sum_{k=1}^d \frac{v_k}{v_d} \sqrt{\frac{a_k}{a_d}} x_k\right) d^{d-1}x\\
&= (2\pi)^\frac{d-1}{2} \int_{-\infty}^\infty e^{-\frac12 y_0^2} \mathrm{erfc}\left(\frac1{\sqrt 2}\sum_{k=1}^d \frac{b_k a_k v_k}{v_d\sqrt{a_d}}-\frac1{\sqrt 2}\sqrt{\sum_{k=1}^{d-1} \frac{a_k v_k^2}{a_d v_d^2}} y_0\right) dy_0.
\end{align*}
The lemma now follows from the identity
\[
\int_{-\infty}^\infty e^{-t^2} \mathrm{erfc}(\alpha+\beta t) \, dt = \sqrt\pi \mathrm{erfc}\left(\frac{\alpha}{\sqrt{1+\beta^2}}\right),
\]
which can be proved, e.g., by differentiating with respect to $\alpha$. Indeed, we have
\[
\frac{(v_d \sqrt{a_d})^{-1}\sum_{k=1}^d b_k a_k v_k}{\sqrt{1+(v_d \sqrt{a_d})^{-2}\sum_{k=1}^{d-1} a_k v_k^2}}
= \frac{\sum_{k=1}^d b_k a_k v_k}{\sqrt{\sum_{k=1}^{d} a_k v_k^2}}
= \frac{b\cdot Av}{\sqrt{v\cdot Av}}.
\]
\end{proof}

\end{appendices}

\end{document}